\documentclass[11pt]{article}
\usepackage{fullpage}
\usepackage{amsmath,amssymb,amsthm}
\usepackage[algo2e,linesnumbered,ruled]{algorithm2e}
\usepackage{graphicx}

\newtheorem{assumption}{Assumption}
\newtheorem{theorem}{Theorem}
\newtheorem{lemma}{Lemma}
\newtheorem{remark}{Remark}
\newtheorem{corollary}{Corollary}

\newcommand{\maximize}{\mathop{\textrm{maximize}}}
\newcommand{\minimize}{\mathop{\textrm{minimize}}}
\newcommand{\xk}{x^k}
\newcommand{\xke}{x^{k+1}}
\newcommand{\xki}{x^k_i}
\newcommand{\xkz}{x^k_0}
\newcommand{\xkt}{x^k_t}
\newcommand{\zki}{z^k_i}
\newcommand{\hxkie}{\hat{x}^{k}_{i+1}}
\newcommand{\xkie}{x^{k}_{i+1}}
\newcommand{\Dki}{\Delta_i^k}
\newcommand{\xkim}{x^{k}_{i-1}}
\newcommand{\Gki}{G_i^k}
\newcommand{\cGki}{\mathcal{G}(\xki)}
\newcommand{\tcGki}{\widetilde{\mathcal{G}}(\xki)}
\newcommand{\tcGkt}{\widetilde{\mathcal{G}}(\xkt)}

\newcommand{\tgk}{\tilde{g}^k}
\newcommand{\tJk}{\tilde{J}^k}
\newcommand{\tgi}{\tilde{g}_i}
\newcommand{\tJi}{\tilde{J}_i}
\newcommand{\tgki}{\tilde{g}^k_i}
\newcommand{\tJki}{\tilde{J}^k_i}
\newcommand{\tgkz}{\tilde{g}^k_0}
\newcommand{\tJkz}{\tilde{J}^k_0}
\newcommand{\tgkim}{\tilde{g}^k_{i-1}}
\newcommand{\tJkim}{\tilde{J}^k_{i-1}}
\newcommand{\Bki}{{\mathcal{B}^k_i}}
\newcommand{\Ski}{{\mathcal{S}^k_i}}
\newcommand{\Bkr}{{\mathcal{B}^k_r}}
\newcommand{\Skr}{{\mathcal{S}^k_r}}
\newcommand{\Bkz}{{\mathcal{B}^k_0}}
\newcommand{\Skz}{{\mathcal{S}^k_0}}
\newcommand{\R}{\mathbf{R}} 
\newcommand{\E}{\mathbf{E}}

\newcommand{\cB}{\mathcal{B}}

\newcommand{\cG}{\mathcal{G}}

\newcommand{\cS}{{\mathcal{S}}}
\newcommand{\cO}{{\mathcal{O}}}

\newcommand{\dom}{\mathrm{dom\,}}
\newcommand{\argmin}{\mathop{\rm argmin}}
\newcommand{\half}{\frac{1}{2}}

\title{Stochastic Variance-Reduced Prox-Linear Algorithms for Nonconvex Composite Optimization}

\author{Junyu Zhang\thanks{%
Department of Electrical Engineering, Princeton University, Princeton, NJ 08544, USA. Email: \texttt{junyuz@princeton.edu}}
\and Lin Xiao\thanks{%
Facebook AI Research (FAIR), Seattle, WA 98109, USA. 
Email: \texttt{linx@fb.com}}
}

\date{May 12, 2021}

\begin{document}
\maketitle 

\begin{abstract} 
    We consider the problem of minimizing composite functions of the form $f(g(x))+h(x)$, where~$f$ and~$h$ are convex functions (which can be nonsmooth) and $g$ is a smooth vector mapping. In addition, we assume that $g$ is the average of finite number of component mappings or the expectation over a family of random component mappings.
    We propose a class of stochastic variance-reduced prox-linear algorithms for solving such problems and bound their sample complexities for finding an $\epsilon$-stationary point in terms of the total number of evaluations of the component mappings and their Jacobians. 
    When~$g$ is a finite average of~$N$ components, we obtain sample complexity $\cO(N+ N^{4/5}\epsilon^{-1})$ for both mapping and Jacobian evaluations.
    When $g$ is a general expectation, we obtain sample complexities of $\cO(\epsilon^{-5/2})$ and $\cO(\epsilon^{-3/2})$ for component mappings and their Jacobians respectively. If in addition~$f$ is smooth, then improved sample complexities of $\cO(N+N^{1/2}\epsilon^{-1})$ and $\cO(\epsilon^{-3/2})$ are derived for $g$ being a finite average and a general expectation respectively, for both component mapping and Jacobian evaluations. 


%\subclass{
%        68Q25  % Analysis of algorithms and problem complexity
%\and    68W20  % Randomized algorithms
%\and    90C15  % Stochastic programming
%\and    90C26  % Nonconvex programming, global optimization
%\and    90C30   % Nonlinear programming
%}
\end{abstract}

\section{Introduction}
\label{sec:intro} 

We consider composite optimization problems of the form
\begin{equation}
\label{eqn:composite-opt}
\minimize_{x\in\R^n} \quad f(g(x)) + h(x), 
\end{equation}
where $f:\R^m\rightarrow\R$ is a convex and possibly nonsmooth function, $g:\R^n\rightarrow\R^m$ is a smooth mapping (vector-valued function), and $h:\R^n\rightarrow\R$ is a convex and lower-semicontinuous function. 
Although both $f$ and $h$ are convex, the problem is in general nonconvex due to the composition of~$f$ and~$g$. 
In addition, we assume that~$g$ is either the average of finite number of component mappings, i.e., $g(x) = \frac{1}{N}\sum_{i=1}^N g_i(x)$, 
or the expectation of a family of random component mappings, i.e.,
$g(x) = \E_\xi [g_\xi(x)]$ where~$\xi$ is a random variable.
More explicitly, we consider the problems
\begin{equation}
\label{eqn:composite-finite}
\minimize_{x\in\R^n} \quad f\biggl(\frac{1}{N}\sum_{i=1}^N g_i(x)\biggr) + h(x)
\end{equation}
and 
\begin{equation}
\label{eqn:composite-expect}
    \minimize_{x\in\R^n} \quad f\bigl(\E_\xi [g_\xi(x)]\bigr) + h(x).
\end{equation}
Clearly, problem~\eqref{eqn:composite-finite} is a special case of~\eqref{eqn:composite-expect} where the random variable~$\xi$ follows the uniform distribution over the finite set $\{1,2,\ldots,N\}$.
We consider them separately because the sample complexity for solving problem~\eqref{eqn:composite-finite} can be much lower than that of the general case~\eqref{eqn:composite-expect}.

An effective method for solving the composite optimization problem~\eqref{eqn:composite-opt} is the (deterministic) \emph{prox-linear} algorithm (e.g., \cite{Prox-Linear-Errbounds-Dmitriy-2018,Gauss-Newton-Nesterov}, which iteratively minimizes a model of the objective function where $g(x)$ is replaced by a linear approximation.
Specifically, let $g':\R^n\to\R^{m\times n}$ denote the Jacobian of~$g$, then
each iteration of prox-linear algorithm takes the form 
\begin{equation}
\label{eqn:prox-linear}
\xke = \argmin_x \left\{f\bigl(g(\xk) + g'(\xk)(x-\xk)\bigr) + h(x) + \frac{M}{2}\|x-\xk\|^2\right\},
\end{equation} 
where $M>0$ is a parameter to penalize the deviation of $\xke$ from~$\xk$ in squared Euclidean distance.
Since~$f$ and~$h$ are convex, the subproblem in~\eqref{eqn:prox-linear} is a convex optimization problem.
For the algorithm to be efficient in practice, we also need the functions~$f$ and~$h$ to be relatively simple, meaning that the subproblem in~\eqref{eqn:prox-linear} admits a closed-form solution or can be solved efficiently. 

For problems~\eqref{eqn:composite-finite} and~\eqref{eqn:composite-expect},
the finite-average and expectation structure of~$g$ allow us to use a randomly sampled subset of $g_i$ or $g_\xi$ and their Jacobians to approximate the expectations~$g$ and~$g'$.
Specifically, during each iteration~$k$, let $\cB^k$ and $\cS^k$ be two subsets of $\{1,2,\ldots,N\}$ sampled uniformly at random or two sets of realizations of $\xi$ sampled from its distribution.
A straightforward approach is to construct the mini-batch approximations
\begin{equation}\label{eqn:mini-batch}
    \tgk =\frac{1}{|\cB^k|}\sum_{i\in\cB^k} g_i(\xk), \qquad
    \tJk =\frac{1}{|\cS^k|}\sum_{i\in\cS^k} g'_i(\xk),
\end{equation}
and use them to replace $g(\xk)$ and $g'(\xk)$ in~\eqref{eqn:prox-linear}, leading to the \emph{stochastic prox-linear} algorithm:
\begin{equation}
\label{eqn:stoch-prox-linear}
\xke = \argmin_x \left\{f\bigl(\tgk + \tJk(x-\xk)\bigr) + h(x) + \frac{M}{2}\|x-\xk\|^2\right\}.
\end{equation} 
While each iteration of~\eqref{eqn:stoch-prox-linear} uses less samples of $g_\xi$ and $g'_\xi$ than the full-batch method~\eqref{eqn:prox-linear}, the simple mini-batch construction in~\eqref{eqn:mini-batch} may not be able to reduce the overall sample complexity due to increased number of iterations required
(see, e.g., \cite{DekelGSX2012} and \cite[Section~3]{NestedSpider}).

In this paper, we develop a class of stochastic \emph{variance-reduced} prox-linear algorithms for solving problems~\eqref{eqn:composite-finite} and~\eqref{eqn:composite-expect}.
By leveraging the variance reduction techniques of SVRG \cite{SVRG-TongZhang,SVRG-LinXiao} and SARAH/\textsc{Spider} \cite{SARAH-1,fang2018spider}, we obtain significantly lower sample complexities than that of the full-batch prox-linear method.
Before getting to the details, we first present several applications.

\subsection{Application examples} 
\label{sec:examples}

Composite optimization problems of the forms~\eqref{eqn:composite-finite} and~\eqref{eqn:composite-expect} arise from risk-averse optimization 
(e.g, \cite{Rockafellar2007CoherentRisk,Ruszczynski2013risk-averse} and a mean-variance tradeoff example in \cite{C-SAGA}) 
and stochastic variational inequalities 
(e.g., \cite{IusemJofre2017,KoshalNedicShanbhag2013},
through a reformulation in~\cite{GhadimiRuszWang2018}).
In machine learning, 
a well-known example is policy evaluation for reinforcement learning
(e.g., \cite{dann2014policy,sutton1998reinforcement,SCGD-M.Wang,ASC-PG-M.Wang}).
Here we give several additional examples, and explain how the stochastic prox-linear algorithms can be applied.

\paragraph{Systems of nonlinear equations for ERM}
Solving systems of nonlinear equations is one of the most fundamental problems in computational science and engineering (e.g., \cite{OrtegaRheinboldt1970}). 
Given a system of nonlinear equations $g(x)=0$ where $g:\R^n\to\R^m$ is a smooth mapping, a standard approach is to minimize the composite function $f(g(x))$ 
where~$f$ is non-negative merit function and $f(z)=0$ if only if $z=0$.
A popular choice is the squared Euclidean norm $f(\cdot)=\|\cdot\|^2$.
The classical Gauss-Newton method iteratively minimizes a simple model by 
linearizing~$g$ at $x^k$:
\[
    x^{k+1} = \argmin_{x} ~\bigl\|g(x^k)+g'(x^k)(x-x^k)\bigr\|^2.
\]
Nesterov \cite{Gauss-Newton-Nesterov} proposed a modified scheme with sharp merit functions such as $f(\cdot)=\|\cdot\|$ and a quadratic penalty term as in~\eqref{eqn:prox-linear}. 
%These methods can serve as effective alternatives to Newton's method 
For empirical risk minimization (ERM) problems of the form 
\[
    \minimize_x \quad F(x)\triangleq \frac{1}{N} \sum_{i=1}^N F_i(x),
\]
where each $F_i$ is twice differentiable, we can apply Gauss-Newton type of methods by letting $g_i(x)=F'_i(x)$ and $g'(x)=F''_i(x)$ (the gradient and Hessian of $F_i$ respectively) and use either a smooth or a sharp merit function~$f$. The resulting optimization problem is of the form~\eqref{eqn:composite-finite} and we can exploit the finite-average structure with the sub-sampled prox-linear algorithm~\eqref{eqn:stoch-prox-linear}. This approach can be particularly useful for solving non-convex ERM problems (see, e.g., \cite{Newton-MR} and \cite{Dingo}). Efficient numerical algorithms for solving the subproblem in each iteration are discussed in \cite{Newton-MR} for $f(\cdot)=\|\cdot\|^2$ and in \cite{Gauss-Newton-Nesterov} for $f(\cdot)=\|\cdot\|$.

\paragraph{Truncated stochastic gradient method}
Consider the stochastic optimization problem
\[
    \minimize_x \quad g(x) \triangleq \E\bigl[g_\xi(x)\bigr],
\]
where each $g_\xi:\R^n\to\R$ is smooth.
Suppose we know the minimum value $g^*=\inf_x g(x)$ or a lower bound of it (in many machine learning problems $g(x)\geq 0$), then the problem is equivalent to
\[
    \minimize_x \quad f(g(x)), \qquad\mbox{where}\quad f(z)=\max\{z,\, g^*\}.
\]
In this case, the mini-batch stochastic prox-linear method~\eqref{eqn:stoch-prox-linear} becomes
\begin{equation}\label{eqn:truncated-sg}
\xke = \argmin_x \left\{ \max\left\{\tgk + \tJk(x-\xk),~g^*\right\} + \frac{M}{2}\|x-\xk\|^2\right\},
\end{equation}
which has a closed-form solution
\[
    x_{k+1} = x_k - \min\left\{\frac{1}{M},\,\frac{\tgk-g^*}{\|\tJk\|^2}\right\}\cdot\tJk.
\]
This update has a very similar step-size rule as Polyak's rule for subgradient method \cite{Polyak87book}.
Because the simple model used in~\eqref{eqn:truncated-sg} truncates the linear model with the known lower bound, it is called the \emph{truncated stochastic gradient method}.
Recent studies \cite{AsiDuchi2019Truncated,AsiDuchi2019StoProxPoint,DamekDima2019ModelBased}
show that it converges faster and is more stable than the classical stochastic gradient method with a wide range of step sizes. 
In this paper, we use variance reduction techniques to construct the estimates
$\tgk$ and $\tJk$ and obtain better sample complexity for this method.

\paragraph{Minimax stochastic optimization}
Consider the problem of minimizing the maximum of $m$ expectations:
\[
    \minimize_{x\in\mathcal{X}}~ \max_{1\leq j\leq m} g^{(j)}(x), \qquad
    \mbox{where}\quad g^{(j)}(x)=\E_{\xi_j}\bigl[g^{(j)}_{\xi_j}(x)\bigr].
\]
Here we assume that $\mathcal{X}$ is a closed convex set and the random variables $\xi_i$ follow (slightly) different probability distributions. 
This is a special case of \emph{distributionally robust optimization}
(see \cite{RahimianMehrotra19DRO} and references therein), which has many applications in operations research and statistical machine learning.
It can be put into the form of~\eqref{eqn:composite-expect} with the definitions
$\xi=[\xi_1,\ldots,\xi_m]$ and
\[
    f(z)=\max_{1\leq j\leq m}z_j, \qquad
    g_{\xi}(x) = \bigl[g^{(1)}_{\xi_1}(x),\ldots,g^{(m)}_{\xi_m}(x)\bigr],\qquad
    h(x) = \delta_{\mathcal{X}}(x),
\]
where $\delta_{\mathcal{X}}$ denotes the indicator function of~$\mathcal{X}$.
In this case, the update in~\eqref{eqn:stoch-prox-linear} requires solving a convex quadratic programming problem.
Similar formulations may apply to other distributionally robust optimization problems.

\paragraph{Exact penalty method for stochastic optimization}
Consider the following constrained stochastic optimization problem
\begin{eqnarray*}
& \displaystyle\minimize_{x\in\mathcal{X}} & \E_{\xi_0}\bigl[g^{(0)}_{\xi_0}(x)\bigr] \\
&\mbox{subject to} & \E_{\xi_j}\bigl[g^{(j)}_{\xi_j}(x)\bigr] \geq 0, \quad j=1,\ldots, m_I,\\ 
& & \E_{\xi_j}\bigl[g^{(j)}_{\xi_j}(x)\bigr] = 0, \quad j=m_I+1,\ldots, m.
\end{eqnarray*}
Using the exact penalty approach (see, e.g., \cite{ExactPenalty-2,ExactPenalty-1}),  this problem can be reformulated as
\begin{align*}
    \minimize_x \quad & \E_{\xi_0}\bigl[g^{(0)}_{\xi_0}(x)\bigr] 
    + \sum_{j=1}^{m_I}c_j\max\left\{0,\,\E_{\xi_j}\bigl[g^{(j)}_{\xi_j}(x)\bigr]\right\} \\ 
    & + \sum_{j=m_I+1}^m c_j\left|\E_{\xi_j}\bigl[g^{(j)}_{\xi_j}(x)\bigr]\right| + \delta_\mathcal{X}(x),
\end{align*}
where $c_j>0$ for $j=1,\ldots,m$ are sufficiently large positive constants
(to ensure the penalty terms vanish at optimality).
It is straightforward to rewrite the above problem as~\eqref{eqn:composite-expect} and we omit the details.
The update in~\eqref{eqn:stoch-prox-linear} also requires solving a convex quadratic programming problem.

\subsection{Related work}
\label{sec:related-work}

The deterministic composite optimization problem \eqref{eqn:composite-opt} is a classical problem in nonconvex and nonsmooth optimization, and its study can date back to the late 70s in the last century; see, e.g., \cite{Composite:Bertsekas,Composite:Fletcher,Composite:Poljak}. 
Recently, there has been a renewed interest in such problems due to many emerging applications,  including the robust phase retrieval problem considered in \cite{PhaseRetrieval-Duchi-2017}, the low-rank semidefinite programming (SDP) problem considered in \cite{LRSDP}, and the robust blind deconvolution problem considered in \cite{Deconvolution}, and so on. 
In fact, many of these applications involve the average or expectation over large amount of component loss functions, similar to those shown in problems~\eqref{eqn:composite-finite} and~\eqref{eqn:composite-expect}.  

For solving the nonlinear least-square problems (when $f=\|\cdot\|^2$), the idea of linearizing the inner mapping~$g$ is well-known from the classical Gauss-Newton method (e.g, \cite[Section~10.3]{NocedalWright2006book}).
For nonsmooth~$f$, the trial of linearizing the inner mapping $g$ was made in
\cite{Prox-Linear:Burke-1985,Prox-Linear:Burke-Ferris-1995}, where the
linearization is used to construct a descent direction for line-search. 
In \cite{Gauss-Newton-Nesterov}, Nesterov proposed the Gauss-Newton type of algorithm \eqref{eqn:prox-linear} for nonsmooth~$f$, analyzed its general convergence properties and proved local quadratic convergence under a non-degeneracy assumption. 
More recently, it has received more attention under the name of prox-linear algorithm. 
The authors of \cite{Prox-Linear-Cartis-2011,Prox-Linear-Dmitriy-2016,Prox-Linear-Ochs-2017} discussed its iteration complexity and the numerical cost of solving the subproblem in each iteration. 
In \cite{Prox-Linear-ProxRivisited-Dmitriy-2017,Prox-Linear-Errbounds-Dmitriy-2018},
the authors studied its fast local convergence property under the
quadratic growth or the error-bound conditions. 
Additional references can be found in
\cite{Burke-1,Burke-2,Prox-Linear-StephenW-2016}. 

In the stochastic settings, it is worth noting that \cite{DamekDima2019ModelBased,Sto-Prox-LR-3,Sto-Prox-LR-2,HazanSabachVoldman2020} have considered the problem
\[
    \minimize_{x\in\mathcal{X}}\quad \E_{\xi}\bigl[f_\xi(g_\xi(x))\bigr],
\] 
where the expectation is taken outside of the composition (in many cases~$f$ does not depend on the random variable~$\xi$).
This problem is essentially a special case of the classical stochastic programming problem. The problems we consider in~\eqref{eqn:composite-finite} and~\eqref{eqn:composite-expect} are quite different.

Algorithms for solving stochastic composite optimization problems of the
forms~\eqref{eqn:composite-finite} and~\eqref{eqn:composite-expect} have been
studied recently in
\cite{blanchet2017unbiased,VRSC-PG,SVR-SCGD,Min-Max,SCGD-M.Wang,ASC-PG-M.Wang,SVR-ADMM,C-SAGA,C-SARAH}.
Since these are all stochastic or randomized algorithms, 
a common measure of performance is their sample complexity, i.e., the total number of samples of the component mappings $g_i$ or $g_\xi$ and their Jacobians required to output some point~$\bar{x}$ such that 
$\E\bigl[\|\cG(\bar{x})\|^2\bigr]\leq\epsilon$, where $\epsilon$ is a predefined precision and $\cG(\bar{x})$ is the composite gradient mapping at~$\bar{x}$
(for a precise definition, see~\eqref{eqn:grad-mapping} in Section~\ref{sec:framework}).
When both~$f$ and~$g$ are smooth and~$g$ is a finite-average, the best sample complexity is $\cO(N+N^{1/2}\epsilon^{-1})$ given in \cite{C-SARAH}, which matches the best known complexity for nonconvex finite-sum optimization without composition \cite{fang2018spider,SARAH-3,ProxSARAH,SpiderBoost}. 
When both~$f$ and~$g$ are smooth and~$g$ is a general expectation, the state-of-the-art sample complexity is the $\cO(\epsilon^{-3/2})$ obtained in \cite{C-SARAH}. 
When $f$ is convex but nonsmooth and $g$ is a finite sum of~$N$ smooth mappings, the authors of \cite{Min-Max} applied the conjugate function of~$f$ and transformed problem \eqref{eqn:composite-finite} to a min-max saddle-point problem. 
The sample complexity of their method (without counting subproblem cost) is $\cO(N\epsilon^{-1})$.

After the initial submission of this paper, we were brought to
attention the independent work \cite{tran2020stochastic}. The authors also consider problems~\eqref{eqn:composite-finite} and~\eqref{eqn:composite-expect} and develop stochastic Gauss-Newton methods (same form as prox-linear algorithms) using SARAH \cite{SARAH-1} for variance reduction. They obtained sample complexity $\cO(\epsilon^{-5/2})$ for $g_\xi$ and $\cO(\epsilon^{-3/2})$ for $g'_\xi$, but for a slightly different stationarity measure than the one used in this paper. We will comment on the connections to our results at the ends of Sections~\ref{sec:nonsmooth-finite} and~\ref{sec:nonsmooth-expect}.

\subsection{Contributions and outline} 
In this paper, we develop a class of stochastic \emph{variance-reduced} prox-linear algorithms for solving problems~\eqref{eqn:composite-finite} and~\eqref{eqn:composite-expect}, by constructing the estimates~$\tgk$ and~$\tJk$ in~\eqref{eqn:stoch-prox-linear} with the variance reduction techniques of SVRG \cite{SVRG-TongZhang,SVRG-LinXiao} and SARAH/\textsc{Spider} \cite{SARAH-1,fang2018spider}. 
Our main results are summarized below.
\begin{itemize}
    \item When $f$ is convex and nonsmooth and $g$ is a finite average, we construct an SVRG type estimator augmented with additional first-order correction, and obtain the sample complexity $\cO(N+N^{4/5}\epsilon^{-1})$ for both component mapping ($g_i$) and Jacobian ($g'_i$) evaluations. 
    \item When $f$ is convex and nonsmooth and $g$ is an expectation of random smooth mappings, we use the SARAH/\textsc{Spider} estimator, and obtain a sample complexity of $\cO(\epsilon^{-5/2})$ for the random mappings ($g_\xi$) and $\cO(\epsilon^{-3/2})$ for the Jacobians ($g'_\xi$).
    \item When $f$ is smooth, we also adopt the SARAH/\textsc{Spider} estimator. For both component mapping and Jacobian evaluations, we obtain the sample complexities $\cO(N + \sqrt{N}\epsilon^{-1})$ and $\cO(\epsilon^{-3/2})$ for the finite average case and expectation case respectively.
\end{itemize}
The first result above (with nonsmooth~$f$ and finite-sum $g$) appears to be new and our sample complexity improves over the best known in the literature \cite{Min-Max}.
The second result is among the first in the literature to derive improved sample complexity for nonsmooth~$f$ and with~$g$ being an expectation (see also~\cite{tran2020stochastic}).
These results can be extended to the cases when $f$ is \emph{weakly convex} (see its definition in, e.g., \cite{DamekDima2019ModelBased,Prox-Linear-Dmitriy-2016}). We omit details to keep the presentation relatively simple, but will make remarks on the necessary changes where it is applicable. 

Note that most work on stochastic composite optimization (SCO) construct the gradient estimators based on chain-rule (see e.g. \cite{SCGD-M.Wang,C-SARAH,NestedSpider}), and they all fail when $f$ is nonsmooth. The significance of our results (and those in \cite{tran2020stochastic}) is to show that using the prox-linear framework, instead of the chain-rule, can take advantage of variance reduction techniques in the nonsmooth composite setting to achieve better sample complexity. 
Another feature that distinguishes our first two results from the existing smooth SCO literature is the imbalance between the required estimation accuracy of $\tilde{g}$ and $\tilde{J}.$ Unlike the chain rule based algrithms for smooth SCO problems where the required accuracy for $\tilde{g}$ and $\tilde{J}$ are of the same order (see e.g. \cite{VRSC-PG,C-SAGA,C-SARAH}), the nonsmooth SCO problem requires the \emph{order} of estimation accuracy for $\tilde{g}$ to be much higher than $\tilde{J}$. New techniques are required to handle this challenge.

Our results with~$f$ being smooth match those in \cite{C-SARAH}, which are obtained by using variance-reduced gradient estimators based on the chain rule, i.e.,
$(\tJk)^Tf'(\tgk)$, in contrast to using the proximal mapping of~$f$ in~\eqref{eqn:stoch-prox-linear}. It is often observed in practice that algorithms based on proximal mappings can be more efficient than those based on gradients, even though in theory they have the same sample complexity (e.g., \cite{AsiDuchi2019Truncated,AsiDuchi2019StoProxPoint,DamekDima2019ModelBased}). Therefore it is very meaningful to establish the convergence and complexity of proximal-mapping based methods even when~$f$ is smooth.
In addition, we comment on its effectiveness by relating to the classical Gauss-Newton method at the end of this paper.

\paragraph{Organization} 
In Section~\ref{sec:framework}, we present a general framework of stochastic variance-reduced prox-linear algorithms using the update formula~\eqref{eqn:stoch-prox-linear}, without specifying how the estimates $\tgk$ and $\tJk$ are constructed. 
In Sections~\ref{sec:nonsmooth-finite} and~\ref{sec:nonsmooth-expect}, we assume that~$f$ can be nonsmooth, and present the constructions of~$\tgk$ and~$\tJk$ and the resulting sample complexities for solving problems~\eqref{eqn:composite-finite} and~\eqref{eqn:composite-expect} respectively. 
In Sections~\ref{sec:smooth-finite} and~\ref{sec:smooth-expect},  we assume that~$f$ is smooth and present the estimators and the corresponding sample complexities for solving these two problems respectively.
In Section~\ref{sec:experiments}, we present preliminary numerical experiments to demonstrate the effectiveness of the proposed algorithms.
We conclude the paper in Section~\ref{sec:discussions} with further discussions on different variance reduction techniques for stochastic composite optimization.

\section{The algorithm framework}
\label{sec:framework}
In this section, we present a framework of stochastic variance-reduced prox-linear algorithms using the update formula~\eqref{eqn:stoch-prox-linear}.
In order to simplify notations, we define
\begin{equation}\label{eqn:Phi-def}
\Phi(x) \triangleq f(g(x)) + h(x),
\end{equation}
where $g$ is either the average of finite number of component mappings as in problem~\eqref{eqn:composite-finite}, 
or the expectation of a family of random component mappings as in~\eqref{eqn:composite-expect}.
We make the following assumptions throughout the paper.
\begin{assumption}\label{assumption:f-cvx} 
    The function $f:\R^m\to\R\cup\{+\infty\}$ is convex and $\ell_f$-Lipschitz continuous, i.e.,
    \[
        |f(u)-f(v)|\leq\ell_f\|u-v\|, \qquad \forall\, u, v\in \R^m.
    \]
The function $h:\R^n\to\R\cup\{+\infty\}$ is convex and lower semi-continuous. 
\end{assumption} 
\begin{assumption}\label{assumption:g-lip}
    The vector mapping $g:\R^n\to\R^m$ is $\ell_g$-Lipschitz continuous and its Jacobian $g':\R^n\to\R^{m\times n}$ is $L_g$-Lipschitz continuous, i.e., 
    \begin{align*}
        \|g(x) - g(y)\| &\leq \ell_g \|x-y\|, \\
        \|g'(x) - g'(y)\| &\leq L_g \|x-y\|, 
    \end{align*}
for all $x,y\in\dom h$, where $\|\cdot\|$ for matrices denotes the spectral norm.
\end{assumption}
A direct consequence of the Lipschitz condition on $g'$ in Assumption~\ref{assumption:g-lip} is 
\begin{equation}\label{eqn:g-quadratic-bd}
    \bigl\|g(x)-g(y)-g(y)'(x-y)\bigr\| \leq \frac{L_g}{2}\|x-y\|^2.
\end{equation}
(See, e.g., \cite[Theorem~3.2.12]{OrtegaRheinboldt1970}.) 
Throughout the paper, we also assume the objective function $\Phi$ is lower bounded, as stated in the following assumption.
\begin{assumption}\label{assumption:phi-lowerbound}
	There exists $\Phi_*$ such that  $\Phi_* = \inf_x\Phi(x)>-\infty$.
\end{assumption}
Under these assumptions, we have the following result.  
\begin{lemma}\label{lemma:upper-bound} 
Suppose Assumption \ref{assumption:f-cvx} and \ref{assumption:g-lip} hold, then 
for any $x,y\in\dom h$,
\begin{equation}
    f(g(x)) \leq f(g(y) + g'(y)(x-y)) + \frac{\ell_f L_g}{2}\|x-y\|^2.
\end{equation}
\end{lemma}
\begin{proof} 
    By the Lipschitz continuity of $f$ and $g'$, we have
	\begin{eqnarray*}
    f(g(x)) 
    &=&  f\bigl(g(y) + g'(y)(x-y)\bigr) + f(g(x)) - f\bigl(g(y) + g'(y)(x-y)\bigr)\\
    & \leq &  f\bigl(g(y) + g'(y)(x-y)\bigr) + \bigl|f(g(x)) - f\bigl(g(y) + g'(y)(x-y)\bigr)\bigr|\\
	&\leq &f\bigl(g(y) + g'(y)(x-y)\bigr) + \ell_f\bigl\|g(x) - g(y) - g'(y)(x-y)\bigr\|\\
	& \leq & f\bigl(g(y) + g'(y)(x-y)\bigr) + \frac{\ell_f L_g}{2}\|x-y\|^2,
	\end{eqnarray*}
    where the last inequality is due to~\eqref{eqn:g-quadratic-bd}.
\end{proof}

As a result of Lemma~\ref{lemma:upper-bound}, $f(g(y) + g'(y)(x-y)) + h(x) + \frac{M}{2}\|x-y\|^2$ is an upper bound of the objective function $f(g(x)) + h(x)$ as long as $M\geq \ell_fL_g$. 
This is exactly the principle of majorization used in the update~\eqref{eqn:prox-linear}.
In order to exploit the finite-average structure of problem~\eqref{eqn:composite-finite}, we can approximate the full average $g(x^k)$ and $g'(x^k)$ with randomly sampled mini-batch estimators $\tgk$ and $\tJk$ as in~\eqref{eqn:mini-batch}.
%and~\eqref{eqn:stoch-prox-linear}.
For problem~\eqref{eqn:composite-expect}, sampling based methods are the only choices because the full expectations $\E_\xi[\cdot]$ are impossible to evaluate in most cases.
As shown in several previous work (see, e.g., \cite{DekelGSX2012} and \cite[Section~3]{NestedSpider}), the simple mini-batching scheme~\eqref{eqn:mini-batch} 
usually does not reduce the overall sample complexity for problems with similar structure, compared with using the full-batch in the finite-average case and using a single sample in the expectation case.

\begin{algorithm2e}[t]
    \caption{Stochastic variance-reduced prox-linear algorithm} %(framework)}
	\label{algo:SVR-PL} 
	{\bf input:} initial point $x^1_0$, $M>0$, number of outer and inner iterations $K$ and $\tau$.\\
	\For{$k = 1,\ldots,K$}{ 
		\For{$i = 0,\ldots,\tau-1$}{
            \eIf{$i==0$}{
                compute $\tgk_0$ and $\tJk_0$ using \emph{large} batches $\Bkz$ and $\Skz$ respectively. \label{alg-line:gk0-Jk0}
            }{
                compute $\tgki$ and $\tJki$ using \emph{small} batches $\Bki$ and $\Ski$ respectively. \label{alg-line:gki-Jki}
            }
            $\xkie = \displaystyle\argmin_x ~\textstyle\Bigl\{f\bigl(\tgki + \tJki(x-\xki)\bigr) + h(x) + \frac{M}{2}\|x-\xki\|^2\Bigr\}$. \label{alg-line:argmin}
		}
		Set $x_0^{k+1} = x^k_{\tau}.$
	}  
    {\bf output:} choose $x^{k^*}_{i^*}$ from $\{\xki\}_{i = 0,\ldots,\tau-1}^{k = 1,\ldots,K}$ uniformly at random.
\end{algorithm2e}   

In this paper, we propose a class of stochastic variance-reduced prox-linear algorithms, outlined in Algorithm~\ref{algo:SVR-PL}, and shown that they achieve better sample complexities than simple mini-batching. 
Following the celebrated SVRG method \cite{SVRG-TongZhang,SVRG-LinXiao}, 
our framework employs an outer loop of~$K$ stages and an inner loop of~$\tau$ iterations. 
During the first iteration of each inner loop, the mapping and Jacobian approximations $\tgkz$ and $\tJkz$ are computed using relatively large sample batches.
In the rest of inner iterations, they are computed with relatively small sample batches. 
It turns out that different variance-reduced estimators are needed to obtain the best sample complexity under different assumptions on~$f$ and the structure of~$g$.
We will present the details of constructing different estimators and their convergence analysis in the remaining sections of this paper.

In order to characterize the sample complexity of different algorithms, we first define what is an $\epsilon$-stationary point.
For any $x\in\dom h$, we define the proximal point
\[
x_+\triangleq
\argmin_y \left\{ f\bigl(g(x)+g'(x)(y-x))+h(y) + \frac{M}{2}\|y-x\|^2\right\}
\]
and the composite gradient mapping at~$x$,
\begin{equation}
\label{eqn:grad-mapping}
\cG_M(x)\triangleq M(x-x_+).
\end{equation}
Given any $\epsilon>0$, we call $\bar{x}$ an $\epsilon$-stationary point of~$\Phi$ defined in~\eqref{eqn:Phi-def} if $\|\cG_M(\bar{x})\|^2\leq\epsilon$.
Note that when $h=0$ and $f$ is the identity mapping, we have $\cG_M(x) = \nabla \Phi(x)$ for any $M>0$ and the definition of $\epsilon$-stationary point reduces to its classical form $\|\nabla\Phi(x)\|^2\leq\epsilon$ for smooth optimization. 
For the validity of $\|\cG_M(\cdot)\|^2$ as an optimality measure under nontrivial~$h$ and nonsmooth~$f$, the readers are referred to \cite{Prox-Linear-Errbounds-Dmitriy-2018}. 
To simplify notation, we will omit the subscript~$M$ (which is a constant throughout this paper) and denote the composite gradient mapping as $\cG(x)$.

The sample complexity of a randomized algorithm, such as Algorithm~\ref{algo:SVR-PL}, is the total number of evaluations of the component mappings $g_i$ or $g_\xi$ and their Jacobians required in order to output some $\bar{x}$ satisfying
\begin{equation}
\label{eqn:eps-solu}
\E\bigl[\|\cG(\bar{x})\|^2\bigr]\leq \epsilon,
\end{equation}
where the expectation is taken over all the random samplings during the iterations of the algorithm.

Notice that the proximal point $x_+$ used in the definition of $\cG(x)$ is computed with~$g(x)$ and~$g'(x)$, which can be very costly if not impossible to evaluate. 
In Algorithm~\ref{algo:SVR-PL}, the proximal point $\xkie$ is computed using the estimates $\tgki$ and $\tJki$, i.e.,
\begin{equation}\label{eqn:xkie-def}
    \xkie = \argmin_x \left\{f\bigl(\tgki + \tJki(x-\xki)\bigr) + h(x) + \frac{M}{2}\|x-\xki\|^2\right\}.
\end{equation}
This leads to a convenient approximation, 
\begin{equation}\label{eqn:approx-gm}
\widetilde{\cG}(\xki)\triangleq M(\xki-\xkie),
\end{equation}
of the true gradient mapping $\cGki= M(\xki-\hat{x}^k_{i+1})$, where 
\begin{equation}\label{eqn:x-hat-def}
    \hat{x}^k_{i+1} %= (\xki)_+
    = \argmin_x \left\{f\bigl(g(\xki) + g'(\xki)(x-\xki)\bigr) + h(x) + \frac{M}{2}\|x-\xki\|^2\right\}.
\end{equation}
Since the definitions of $\epsilon$-stationary point and sample complexity are based on the true gradient mapping $\cG$ but computationally we only have access to the approximation $\widetilde{\cG}$, we need to derive a bound between them for the purpose of complexity analysis.
Not surprisingly, such a bound depends on the approximation quality of the estimators $\tgki$ and $\tJki$, as shown in the following lemma.

\begin{lemma}
\label{lemma:PG-finite-sum-nonsmooth}
Under Assumptions \ref{assumption:f-cvx} and \ref{assumption:g-lip}, 
the iterates generated by Algorithm~\ref{algo:SVR-PL} satisfy
\begin{eqnarray*}
    \frac{M - \ell_f L_g}{M^2}\bigl\|\cG(\xki)\bigr\|^2 
    &\leq& \frac{2M+\ell_f L_g}{M^2}\bigl\|\widetilde{\cG}(\xki)\bigr\|^2  \\
    && \,+\, 4\ell_f\bigl\|\tgki-g(\xki)\bigr\| + \frac{2\ell_f}{L_g}\bigl\|\tJki-g'(\xki)\bigr\|^2.
\end{eqnarray*}
\end{lemma}
\begin{proof}
For the ease of notation,  we denote 
\begin{eqnarray}
    F(x;\xki) &=& f\bigl(g(\xki) + g'(\xki)(x-\xki)\bigr), \label{eqn:F-def}\\
    \widetilde{F}(x;\xki) &=& f\bigl(\tgki + \tJki(x-\xki)\bigr).
    \label{eqn:tF-def}
\end{eqnarray}
Since both~$f$ and~$h$ are convex (Assumption~\ref{assumption:f-cvx}),
the following two functions are $M$-strongly convex:
\begin{eqnarray}
    && F(x;\xki) + h(x) + \frac{M}{2}\|x-\xki\|^2, \label{eqn:F+h+M} \\
    && \widetilde{F}(x;\xki) + h(x) + \frac{M}{2}\|x-\xki\|^2.
    \label{eqn:tF+h+M} 
\end{eqnarray} 
According to~\eqref{eqn:x-hat-def} and~\eqref{eqn:xkie-def}, $\hat{x}_{i+1}^k$ and $\xkie$ are the minimizers of these two functions respectively. 
Therefore
\begin{eqnarray*}
    F(\hat{x}^k_{i+1};\xki) + h(\hat{x}^k_{i+1}) + \frac{M}{2}\|\hat{x}^k_{i+1}-\xki\|^2 
	& \leq & F(\xkie;\xki) + h(\xkie) + \frac{M}{2}\|\xkie-\xki\|^2  \\
    &&  - \frac{M}{2}\|\hat{x}^k_{i+1}-\xkie\|^2,  
\end{eqnarray*} 
and
\begin{eqnarray*}
    \widetilde{F}(x^k_{i+1};\xki) + h(x^k_{i+1}) + \frac{M}{2}\|x^k_{i+1}-\xki\|^2 
& \leq & \widetilde{F}(\hat{x}^k_{i+1};\xki) + h(\hat{x}^k_{i+1}) + \frac{M}{2}\|\hat{x}^k_{i+1}-\xki\|^2 \\
&& - \frac{M}{2}\|\hat{x}^k_{i+1}-\xkie\|^2. 
\end{eqnarray*} 
Summing the two inequalities above and rearranging the terms, we obtain
\begin{eqnarray}
	\label{lm:PG-finite-sum-3}
	M\|\hat{x}^k_{i+1}-\xkie\|^2 
    &\leq& F(\xkie;\xki)-\widetilde{F}(\xkie;\xki) 
     + \widetilde{F}(\hat{x}^k_{i+1};\xki)-F(\hat{x}^k_{i+1};\xki).
\end{eqnarray}
Using the Lipschitz property of~$f$, we have 
\begin{eqnarray*}
\bigl| F(\xkie;\xki)-\widetilde{F}(\xkie;\xki)\bigr| 
    &=& \left| f\bigl(g(\xki)+g'(\xki)(\xkie-\xki)\bigr) - f\bigl(\tgki+\tJki(\xkie-\xki)\bigr)\right| \\
    & \leq & \ell_f\left\|\bigl(g(\xki) - \tgki\bigr) + \bigl(g'(\xki) - \tJki\bigr)\bigl(\xkie - \xki\bigr)\right\|\\
	& \leq & \ell_f\left(\bigl\|\tgki - g(\xki)\bigr\| + \bigl\|\tJki - g'(\xki)\bigr\|\bigl\|\xkie - \xki\bigr\|\right)\\
    & \leq & \ell_f\left(\bigl\|\tgki - g(\xki)\bigr\| + \frac{1}{2L_g}\bigl\|\tJki - g'(\xki)\bigr\|^2 + \frac{L_g}{2}\bigl\|\xkie-\xki\bigr\|^2 \right).
\end{eqnarray*}
Replacing $\xkie$ in the above inequality with $\hxkie$, we get  
\begin{eqnarray*}
\bigl|F(\hxkie;\xki)-\widetilde{F}(\hxkie;\xki)\bigr| 
	& \leq & \ell_f \left(\bigl\|\tgki - g(\xki)\bigr\| + \frac{1}{2L_g}\bigl\|\tJki - g'(\xki)\bigr\|^2 + \frac{L_g}{2}\bigl\|\hxkie-\xki\bigr\|^2\right).
\end{eqnarray*}
Combining the two bounds above with~\eqref{lm:PG-finite-sum-3} gives
\begin{eqnarray*}
	M\bigl\|\hxkie-\xkie\bigr\|^2 
    & \leq & 2\ell_f\bigl\|\tgki - g(\xki)\bigr\| + \frac{\ell_f}{L_g}\bigl\|\tJki - g'(\xki)\bigr\|^2 \\
    && +\, \frac{\ell_f L_g}{2}\bigl\|\xkie-\xki\bigr\|^2+\frac{\ell_f L_g}{2}\bigl\|\hxkie-\xki\bigr\|^2 .
\end{eqnarray*}
Next, using the fact that $\|a+b\|^2\leq 2\|a\|^2+2\|b\|^2$ and the above inequality, we have
\begin{eqnarray*} 
	M\bigl\|\hxkie-\xki\bigr\|^2] 
    & \leq & 2M\bigl\|\xkie-\xki\bigr\|^2 + 2M\bigl\|\hxkie -\xkie\bigr\|^2 \\
	& \leq & 2M\bigl\|\xkie-\xki\bigr\|^2 + 4\ell_f\bigl\|\tgki - g(\xki)\bigr\| + \frac{2\ell_f}{L_g}\bigl\|\tJki - g'(\xki)\bigr\|^2 \\
	& & + \ell_f L_g\bigl\|\xkie-\xki\bigr\|^2 + \ell_f L_g\bigl\|\hxkie-\xki\bigr\|^2.
\end{eqnarray*}
Rearranging the terms yields 
\begin{eqnarray*}
    (M - \ell_f L_g)\bigl\|\hxkie-\xki\bigr\|^2 
    &\leq& (2M+\ell_f L_g)\bigl\|\xkie-\xki\bigr\|^2  \\
    && +\, 4\ell_f\bigl\|\tgki-g(\xki)\bigr\| 
     \,+\, \frac{2\ell_f}{L_g}\bigl\|\tJki-g'(\xki)\bigr\|^2.
\end{eqnarray*}
Finally, using the definitions $\cG(\xki)=M(\xki-\hxkie)$ and
$\widetilde{\cG}(\xki)=M(\xki-\xkie)$, we obtain the desired result.
\end{proof}

\paragraph{Extension to the weakly convex case.}
The function~$f$ is $\rho$-weakly convex if $f(x)+\frac{\rho}{2}\|x\|^2$ is convex. In order to extends results in this paper for weakly convex~$f$, we need to increase~$M$ to ensure that the functions in~\eqref{eqn:F+h+M} and~\eqref{eqn:tF+h+M} are strongly convex (in fact, strong convexity in expectation is sufficient).

\section{The nonsmooth and finite-average case}
\label{sec:nonsmooth-finite} 

In this section, we consider the composite finite-average problem~\eqref{eqn:composite-finite} with nonsmooth~$f$ and smooth~$g_i$'s.
In particular, we replace Assumption~\ref{assumption:g-lip} with the following more structured one, which implies Assumption~\ref{assumption:g-lip}.
\begin{assumption}
	\label{assumption:g-lip-finite}
    For each $i = 1,\ldots,N$, the mapping $g_i:\R^n\to\R^m$, is $\ell_{g,i}$-Lipschitz continuous and its Jacobian matrix $g'_i:\R^n\to\R^{m\times n}$ is $L_{g,i}$-Lipschitz continuous. Namely, 
\begin{align*}
\|g_i(x) - g_i(x)\| & \leq\ell_{g,i}\|x-y\|, \\
\|g'_i(x) - g'_i(x)\| & \leq L_{g,i}\|x-y\|,
\end{align*}
for all $x,y\in\dom h$ and $i = 1,\ldots,N$. 
\end{assumption}
A direct consequence of this assumption is that $g$ is $\big(\frac{1}{N}\sum_{i}^{N}\ell_{g,i}\big)$-Lipschitz continuous and $g'$ is $\big(\frac{1}{N}\sum_{i}^{N}L_{g,i}\big)$-Lipschitz continuous. Due to the root-mean square inequality $\frac{z_1+...+z_N}{N}\leq \sqrt{\frac{z_1^2+...+z_N^2}{N}}$, 
We define 
\begin{equation}\label{eqn:ms-Lip-constants}
\ell_g = \sqrt{\frac{1}{N}\sum_{i}^{N}\ell_{g,i}^2}, \qquad
L_g = \sqrt{\frac{1}{N}\sum_{i}^{N}L_{g,i}^2},
\end{equation}
which can serve as the Lipschitz constants of $g$ and $g'$ respectively
as in Assumption~\ref{assumption:g-lip}.

In this case, we construct the estimates $\tgkz$ and $\tJkz$ using the full batch. In other words, we let $\Bkz=\Skz=\{1,2,\ldots,N\}$ and replace Line~\ref{alg-line:gk0-Jk0} in Algorithm~\ref{algo:SVR-PL} with
\begin{eqnarray}
\tgkz &=& g(\xkz) = \frac{1}{N}\sum_{i=1}^N g_i(\xkz), 
    \label{eqn:finite-gk0} \\
    \tJkz &=& g'(\xkz) = \frac{1}{N}\sum_{i=1}^N g'_i(\xkz).
    \label{eqn:finite-Jk0}
\end{eqnarray}
For $i>0$, we sample with replacement from $\{1,2,\ldots,N\}$ to obtain smaller sets $\Bki$ and $\Ski$
(whose cardinalities will be determined later), and
apply the following construction:
\begin{eqnarray}
    \tgki &=& \frac{1}{|\cB_i^k|}\sum_{j\in\cB_i^k} \Bigl(g_j(\xki) - g_j(x_0^k) - g_j'(x_0^k)(\xki-x_0^k)\Bigr) + g(x_0^k) + g'(x_0^k)(\xki-x_0^k),
          \label{eqn:g-svrg-correction} \\
    \tJki &=& \frac{1}{|\cS_i^k|}\sum_{j\in\cS_i^k} \Bigl(g'_j(\xki) - g'_j(x_0^k)\Bigr) + g'(x_0^k). 
\label{eqn:J-svrg} 
\end{eqnarray}
It is worth noting that here we use the standard SVRG estimator \cite{SVRG-TongZhang} to construct $\tJki$, but the estimator for $\tgki$ is augmented with a first-order correction (a similar estimator was proposed in~\cite{ZhouXuGu2018}). 

We remark that for nonsmooth~$f$, the first-order correction scheme in~\eqref{eqn:g-svrg-correction} is essential for achieving a sample complexity that is sublinear in~$N$, whereas purely applying the SVRG estimator will only result in a sample complexity linear in~$N$. This is very different from the case with smooth~$f$ (see e.g. \cite{VRSC-PG,SVR-SCGD}). The main reason for such distinction is that nonsmooth SCO problem requires the estimation accuracy for $\tilde{g}_i^k$ to be much higher than $\tilde{J}_i^k$. 
In addition, the SARAH/\textsc{Spider} estimators seem to be not compatible with the first-order correction technique and we are not able to combine them together in order to obtain a sample complexity that is sublinear in~$N$.

The following lemma bounds the approximation errors of these estimators. 
 \begin{lemma}
	\label{lemma:svr-f-new}
    Suppose Assumption~\ref{assumption:g-lip-finite} holds and $\tgki$ and $\tJki$ are constructed according to~\eqref{eqn:g-svrg-correction} and~\eqref{eqn:J-svrg} respectively, then 
\begin{eqnarray*}
	\E\Bigl[\bigl\|\tgki - g(\xki)\bigr\| \,\big|\, \xki\Bigr] 
        &\leq& \frac{L_g}{2\sqrt{|\cB^k_i|}}\bigl\|\xki - x_0^k\bigr\|^2,\\ 
	\E\Bigl[\bigl\|\tJki - g'(\xki)\bigr\|^2 \,\big|\, \xki\Bigr] 
    &\leq& \frac{L_g^2}{|\cS_i^k|}\bigl\|\xki - x_0^k\bigr\|^2,
\end{eqnarray*}
where $\E[\cdot|\xki]$ denotes conditional expectation given~$\xki$, i.e., expectation with respect to the random indices in $\Bki$ and $\Ski$.
\end{lemma}
\begin{proof}
    To prove the first inequality, we start with~\eqref{eqn:g-svrg-correction} and write
\begin{equation}\label{eqn:tgki-Z}
    \tgki - g(\xki) = \frac{1}{|\Bki|}\sum_{j\in\Bki}Z_j,
\end{equation}
    where
    $$Z_j = g_j(\xki) - g_j(x_0^k) - g_j'(x_0^k)(\xki-x_0^k) \,+\, g(x_0^k) + g'(x_0^k)(\xki-x_0^k) \,-\, g(\xki).$$
    Since~$j$ is randomly sampled from $\{1,2,\ldots,N\}$, we have $\E[g_j(\xki)]=g(\xki)$ and $\E[g'_j(\xki)]=g'(\xki)$, which implies $\E[Z_j|\xki] = 0$. That is, $\tgki$ is an unbiased estimate of $g(\xki)$.
    In addition, we have
	$$\E\bigl[g_j(\xki) - g_j(x_0^k) - g_j'(x_0^k)(\xki-x_0^k)\,\bigl|\,\xki\bigl] =   g(\xki) - g(x_0^k) - g'(x_0^k)(\xki-x_0^k).$$ 
This allows us to bound the variance of $Z_j$ as follows:
\begin{eqnarray*}
	\E\bigl[\|Z_j\|^2\,\big|\,\xki\bigr] 
    & = & \E\Bigl[\bigl\|g_j(\xki) - g_j(x_0^k) - g_j'(x_0^k)(\xki-x_0^k) \bigr\|^2\,\big|\,\xki\Bigr] \\
    && -\, \bigl\|g(\xki) - g(x_0^k) - g'(x_0^k)(\xki-x_0^k)\bigr\|^2 \\
	& \leq & \E\Bigl[\bigl\|g_j(\xki) - g_j(x_0^k) - g_j'(x_0^k)(\xki-x_0^k) \bigr\|^2 \,\big|\,\xki\Bigr]\\
	&\leq & \frac{1}{N}\sum_{j=1}^N\left(\frac{L_{g,j}}{2}\|\xki-x_0^k\|^2\right)^2\\
	& = & \frac{L_{g}^2}{4}\|\xki-x_0^k\|^4,
\end{eqnarray*}
where the last inequality is due to~\eqref{eqn:g-quadratic-bd} and Assumption \ref{assumption:g-lip-finite} respectively.
In the last equality, we used the definition of~$L_g$ in~\eqref{eqn:ms-Lip-constants}.

Combining the above inequality with~\eqref{eqn:tgki-Z} yields
	$$\E\Bigl[\bigl\|\tgki - g(\xki)\bigr\|^2 \,\big|\, \xki\Bigr] \leq \frac{L_g^2}{4|\cB^k_i|}\bigl\|\xki - x_0^k\bigr\|^4.$$ 
    Next, using the concavity of~$\sqrt{\cdot}$ and Jensen's inequality, we obtain the desired result:
\[
	\E\Bigl[\bigl\|\tgki - g(\xki)\bigr\| \,\big|\, \xki\Bigr] 
	\leq \sqrt{\E\Bigl[\bigl\|\tgki - g(\xki)\bigr\|^2 \,\big|\, \xki\Bigr] }
	\leq \frac{L_g}{2\sqrt{|\cB^k_i|}}\bigl\|\xki - x_0^k\bigr\|^2.
\]
To prove the second inequality, we define $Z_j = g'_j(\xki) - g'_j(x_0^k) + g'(x_0^k) - g'(\xki)$ and follow a similar line of arguments. 
\end{proof}

Next, we prove a descent property of the algorithm, which is a crucial step for the convergence analysis.
\begin{lemma}
\label{lemma: descent-Finite-nonsmooth}
Suppose Assumptions~\ref{assumption:f-cvx} and~\ref{assumption:g-lip-finite} hold and the estimates $\tgkz$, $\tJkz$, $\tgki$ and $\tJki$ in Algorithm~\ref{algo:SVR-PL} are constructed as in~\eqref{eqn:finite-gk0}-\eqref{eqn:J-svrg} respectively.
Then for $k=1,\ldots,K$ and $i=0,\ldots,\tau-1$, 
\begin{eqnarray}
\Phi(\xkie) 
&\leq& \Phi(\xki) - \frac{M - 2\ell_f L_g}{2M^2}\bigl\|\tcGki\bigr\|^2 \nonumber \\
&& +\, 2\ell_f \bigl\|\tgki-g(\xki)\bigr\| + \frac{\ell_f}{2L_g}\bigl\|\tJki-g'(\xki)\bigr\|^2. 
\label{eqn:descent-lemma}
\end{eqnarray}
\end{lemma}
\begin{proof}
By the definition of $\Phi$ in~\eqref{eqn:Phi-def} and Lemma~\ref{lemma:upper-bound}, we have 
\begin{eqnarray}
\Phi(\xkie) & \leq & f(g(\xki) + g'(\xki)(\xkie-\xki)) + \frac{\ell_f L_g}{2}\bigl\|\xkie-\xki\bigr\|^2 + h(\xkie) \nonumber \\
& = &  f\bigl(\tgki + \tJki(\xkie-\xki)\bigr) + h(\xkie) + \frac{M}{2}\bigl\|\xkie-\xki\bigr\|^2
 -\, \frac{M - \ell_f L_g}{2}\bigl\|\xkie-\xki\bigr\|^2 \nonumber \\
& & + \underbrace{f\bigl(g(\xki) + g'(\xki)(\xkie-\xki)\bigr)-f\bigl(\tgki + g'(\xki)(\xkie-\xki)\bigr)}_{T_1}\nonumber\\
& & + \underbrace{f\bigl(\tgki + g'(\xki)(\xkie-\xki)\bigr)-f\bigl(\tgki + \tJki(\xkie-\xki)\bigr)}_{T_2} .
\label{eqn:descent-finite-0}
\end{eqnarray}
According to~\eqref{eqn:xkie-def}, we have
\[
    f\bigl(\tgki + \tJki(\xkie-\xki)\bigr) + h(\xkie) + \frac{M}{2}\bigl\|\xkie-\xki\bigr\|^2 ~\leq~ f(\tgki) + h(\xki).
\]
Therefore,
\begin{eqnarray}
\Phi(\xkie) 
&\leq & f(\tgki) + h(\xki) - \frac{M-\ell_f L_g}{2} \|\xkie-\xki\|^2 + T_1+ T_2 \nonumber\\
&\leq & f(g(\xki)) + h(\xki) - \frac{M-\ell_f L_g}{2} \|\xkie-\xki\|^2 + T_1+ T_2  +  \underbrace{f(\tgki)-f(g(\xki))}_{T_3}\nonumber\\
& = & \Phi(\xki)  - \frac{M-\ell_f L_g}{2} \|\xkie-\xki\|^2 + T_1+ T_2 +T_3. 
\label{eqn:Phi-descent}
\end{eqnarray}
By the Lipschitz property of~$f$, we have
\begin{eqnarray*}
    T_1 \leq \ell_f \bigl\|\tgki-g(\xki)\bigr\|, \qquad
    T_3 \leq \ell_f \bigl\|\tgki-g(\xki)\bigr\|,
\end{eqnarray*}
and
\begin{eqnarray}
T_2 &\leq& \ell_f\bigl\|\bigl(\tJki-g'(\xki)\bigr)(\xkie-\xki)\bigr\|
~\leq~ \ell_f\bigl\|\tJki-g'(\xki)\bigr\|\cdot\bigl\|\xkie-\xki\bigr\| \nonumber\\
&\leq& \frac{\ell_f}{2L_g}\bigl\|\tJki-g'(\xki)\bigr\|^2 + \frac{\ell_f L_g}{2}\bigl\|\xkie-\xki\bigr\|^2.
\label{eqn:descent-lemma-T2}
\end{eqnarray}
Combining the bounds on $T_1$, $T_2$ and $T_3$ and the inequality~\eqref{eqn:Phi-descent} yields
\begin{eqnarray*}
    \Phi(\xkie) 
    &\leq& \Phi(\xki) - \left(\frac{M}{2} - \ell_f L_g\right)\bigl\|\xkie-\xki\bigr\|^2 \\
    && +\, 2\ell_f \bigl\|\tgki-g(\xki)\bigr\| + \frac{\ell_f}{2L_g}\bigl\|\tJki-g'(\xki)\bigr\|^2, 
\end{eqnarray*}
which, upon noticing $\tcGki=-M(\xkie-\xki)$, is equivalent to the
desired result. 
\end{proof} 

Recall the definition that $\tcGki: = M(\xki-\xkie)$. 
In order to complete the convergence analysis, we define a stochastic Lyapunov function
\begin{equation}\label{eqn:lyapunov-def}
R_i^k = \E\left[\Phi(\xki) + c_i\biggl\|\sum_{t=0}^{i-1}\tcGkt\biggr\|^2\right],
\quad k=1,\ldots,K, \quad i=0,\ldots,\tau,
\end{equation}
where the coefficients $c_i$ for $i=0,1,\ldots,\tau$ are obtained through the recursion:
\begin{eqnarray}
c_{\tau} &=& 0, \nonumber \\
c_i &=& c_{i+1}\left(1+\frac{1}{\tau}\right) + \frac{1}{3M\sqrt{|\Bki|}} + \frac{1}{5M|\Ski|}, \quad i=\tau-1,\ldots,0.
\label{eqn:ct-finite-nonsmooth}
\end{eqnarray}
(Our choices or $|\Bki|$ and $|\Ski|$ will not depend on~$k$.)
In addition, we define the following constant
\[
   \gamma \triangleq \min_{0\leq i\leq \tau-1} \frac{1}{3}\left(\frac{1}{4M} - c_{i+1}(1+\tau)\right).
\] 
We can ensure $\gamma>0$ by choosing~$\tau$, $|\Bki|$ and $|\Ski|$ appropriately.
We will discuss how to set these values after the following lemma, where we simply assume $\gamma>0$.

\begin{lemma}
	\label{lemma:Lyapunov-finite-nonsmooth}
Suppose Assumptions~\ref{assumption:f-cvx} and~\ref{assumption:g-lip-finite} hold and the estimates $\tgkz$, $\tJkz$, $\tgki$ and $\tJki$ in Algorithm~\ref{algo:SVR-PL} are constructed as in~\eqref{eqn:finite-gk0}-\eqref{eqn:J-svrg} respectively.
In addition, we assume $M\geq 4\ell_f L_g$ and $\gamma>0$.
Then for each $k=1,\ldots,K$,
\begin{equation}\label{lm:Lyapunov-finite-nonsmooth-1}
	\sum_{i=0}^{\tau-1}\E\Bigl[\bigl\|\cGki\bigr\|^2\Bigr] 
    ~\leq~  \frac{R^k_0-R^k_\tau}{\gamma}  
    ~=~ \frac{\E[\Phi(x_0^k)]-\E[\Phi(x_0^{k+1})]}{\gamma}.
\end{equation}
\end{lemma}
\begin{proof}
For the ease of notation, we write the stochastic Lyapunov function as
$$R^k_i = \E\bigl[\Phi(\xki) + c_i\|G_i^k\|^2\bigr],$$
where
\begin{equation}\label{eqn:Gki-def}
    \Gki = \sum_{t=0}^{i-1}\tcGkt = -M(\xki - x_0^k).
\end{equation}
In particular, we have $G_{0}^k = -M(x_0^k-x_0^k) = 0$.
Moreover, we have 
\begin{eqnarray}
	\E\bigl[\|G_{i+1}^k\|^2\bigr] 
    &=& \E\bigl[\|G_{i}^k + \tcGki\|^2\bigr]   \nonumber \\
    &\leq& \left(1+\frac{1}{\tau}\right)\E\bigl[\|G_{i}^k\|^2\bigr] + (1+\tau)\E\bigl[\|\tcGki\|^2\bigr]. 
\label{eqn:Gki-tcG}
\end{eqnarray}
Combining Lemmas~\ref{lemma:PG-finite-sum-nonsmooth} 
and~\ref{lemma: descent-Finite-nonsmooth} yields
\begin{eqnarray*}
\E\bigl[\Phi(\xkie)\bigr] 
    &\leq& \E\bigl[\Phi(\xki)\bigr] - \frac{M - 2\ell_f L_g}{2M^2}\E\bigl[\bigl\|\tcGki\bigr\|^2\bigr] + \left(\frac{\ell_f L_g}{\sqrt{|\cB_i^k|}}+\frac{\ell_f L_g}{2|\cS_i^k|}\right)\E\bigl[\|\xki-x_0^k\|^2\bigr] \\
    &=& \E\bigl[\Phi(\xki)\bigr] - \frac{M - 2\ell_f L_g}{2M^2}\E\bigl[\bigl\|\tcGki\bigr\|^2\bigr] + \frac{1}{M^2}\left(\frac{\ell_f L_g}{\sqrt{|\cB_i^k|}}+\frac{\ell_f L_g}{2|\cS_i^k|}\right)\E\bigl[\|G^k_i\|^2\bigr],
\end{eqnarray*}
where in the last equality we used~\eqref{eqn:Gki-def}.
Adding both sides of~\eqref{eqn:Gki-tcG} to that of the above inequality 
and using the assumption $M\geq 4\ell_f L_g$, we obtain
\begin{eqnarray}
\E\bigl[\Phi(\xkie)+c_{i+1}\|G_{i+1}^k\|^2\bigr] 
    & \leq & \E\bigl[\Phi(\xki)\bigr] - \left(\frac{M-2\ell_f L_g}{2M^2} - c_{i+1}(1+\tau)\right)\E\bigl[\|\tcGki\|^2\bigr] \nonumber \\ 
    & & +\left(\frac{1}{M^2}\left(\frac{\ell_f L_g}{\sqrt{|\cB_i^k|}}+\frac{\ell_f L_g}{2|\cS_i^k|}\right) + c_{i+1}\left(1+\frac{1}{\tau}\right)\right)\E\bigl[\|G_i^k\|^2\bigr] \nonumber \\
    & \leq & \E\bigl[\Phi(\xki)\bigr] - \left(\frac{1}{4M} - c_{i+1}(1+\tau)\right)\E\bigl[\|\tcGki\|^2\bigr] \label{eqn:lyapunov-descent}\\ 
    & & +\left(\frac{1}{4M}\left(\frac{1}{\sqrt{|\cB_i^k|}}+\frac{1}{2|\cS_i^k|}\right) + c_{i+1}\left(1+\frac{1}{\tau}\right)\right)\E\bigl[\|G_i^k\|^2\bigr]. \nonumber
\end{eqnarray}
Next, combining Lemma~\ref{lemma:PG-finite-sum-nonsmooth} with Lemmas~\ref{lemma:svr-f-new} yields
\begin{eqnarray*}
    \frac{M - \ell_f L_g}{M^2}\E\bigl[\|\cGki\|^2\bigr]
    &\leq&
    \frac{2M+\ell_f L_g}{M^2}\E\bigl[\|\tcGki\|^2\bigr] 
     + \left( \frac{2\ell_f L_g}{\sqrt{|\cB_i^k|}} + \frac{2\ell_f L_g}{|\cS_i^k|} \right) \E\bigl[\|\xki-\xkz\|^2\bigr].
\end{eqnarray*}
Using the equality $G^k_i=-M(\xki-\xkz)$ and the assumption $M\geq 4\ell_f L_g$, the above inequality implies
\begin{eqnarray}
    \frac{3}{4M}\E\bigl[\|\cGki\|^2\bigr]
    &\leq&
    \frac{9}{4M}\E\bigl[\|\tcGki\|^2\bigr] 
     +\, \frac{1}{2M} \left( \frac{1}{\sqrt{|\cB_i^k|}} + \frac{1}{|\cS_i^k|} \right) \E\bigl[\|G^k_i\|^2\bigr].
    \label{eqn:cG-bound}
\end{eqnarray}
Multiplying both sides of~\eqref{eqn:cG-bound} by $\left(\frac{1}{4M}-c_{i+1}(1+\tau)\right)\big/\frac{9}{4M}$, which is positive by the assumption $\gamma>0$, and adding the resulting inequality to~\eqref{eqn:lyapunov-descent}, we get
\begin{eqnarray}
%\label{lm:Lyapunov-finite-nonsmooth-2}
\E\bigl[\Phi(\xkie)+c_{i+1}\|G_{i+1}^k\|^2\bigr] 
&\leq& \E\left[\Phi(\xki)+\left( c_{i+1}\left(1+\frac{1}{\tau}\right)+ \frac{1}{3M\sqrt{|\cB_i^k|}} + \frac{1}{5M|\cS_i^k|}\right)\|G_i^k\|^2\right] \nonumber\\
	&& -\frac{1}{3}\left(\frac{1}{4M} - c_{i+1}(1+\tau)\right)\E[\|\cGki\|^2].\nonumber
\end{eqnarray}
Now, using the definitions in~\eqref{eqn:lyapunov-def} and~\eqref{eqn:ct-finite-nonsmooth}, the above inequality is the same as
\[
	\frac{1}{3}\left(\frac{1}{4M} - c_{i+1}(1+\tau)\right)\E\bigl[\|\cGki\|^2\bigr] ~\leq~  R_i^k  - R_{i+1}^k.
\]
Recalling the definition of $\gamma$ and summing up the above inequality over~$i$ from~$0$ to~$\tau-1$, we get
\[
	\gamma\sum_{i=0}^{\tau-1}\E\bigl[\|\cGki\|^2\bigr]
    ~\leq~ R_0^k-R_\tau^k
    ~=~\E\bigl[\Phi(x_0^k)\bigr] - \E\bigl[\Phi(x_\tau^k)\bigr] ,
\]
where the last equality is due to the observations that 
$c_\tau=0$ and $G^k_0=0$. 
Finally, dividing both sides by $\gamma$ and using $x^{k+1}_0 = x^k_\tau$
give the desired result.
\end{proof}

The next lemma shows how to choose the inner loop length~$\tau$ and the two mini-batch sizes $|\Bki|$ and $|\Ski|$ to ensure $\gamma>0$.
We use $\lceil\cdot\rceil$ to denote the nearest integer from above.
\begin{lemma}
\label{lemma:gamma-nonsmooth}
    If we choose 
    $\tau = \bigl\lceil\frac{1}{2}N^{1/5}-1\bigr\rceil$,
    $|\cB_i^k| = \bigr\lceil 4N^{4/5}\bigr\rceil$ and
    $|\cS_i^k| = \bigr\lceil N^{2/5}\bigr\rceil$ for $i=1,\ldots,\tau-1$,
    then $\gamma  \geq \frac{1}{15M}$.
\end{lemma}
\begin{proof}
To simplify notation, let $B=|\Bki|$ and $S=|\Ski|$ for $i=1,\ldots,\tau-1$.
From \eqref{eqn:ct-finite-nonsmooth}, we deduce
\[
    (c_i+C) = (c_{i+1}+C)\left(1+\frac{1}{\tau}\right), \qquad
        \mbox{where}\quad C = \frac{\tau}{3M\sqrt{B}} + \frac{\tau}{5MS}.
\] 
Consequently, with $c_\tau=0$, we have for all $i=1,\ldots,\tau$, 
\[
    c_i ~=~ (c_\tau + C)\left(1+\frac{1}{\tau}\right)^{\tau-i}\!\! -C
    ~\leq~  C\left(1+\frac{1}{\tau}\right)^{\tau} - C
    ~\leq~ Ce -C ~=~ C(e-1),
\] 
where the last inequality is due to the fact that 
$(1+1/\tau)^\tau\leq e$ with $e$ is Euler's number (the basis of natural logarithm). Therefore, 
\begin{eqnarray*}
	\gamma  
    & =  & \min_{0\leq i\leq \tau-1}\frac{1}{3}\left(\frac{1}{4M} - c_{i+1}(1+\tau)\right)\\
    & \geq & \frac{1}{3}\left(\frac{1}{4M} - C(e-1)(1+\tau)\right)\\
    & = & \frac{1}{3M}\left(\frac{1}{4} - \left( \frac{1}{3\sqrt{B}} + \frac{1}{5S}\right)(e-1)\tau(1+\tau)\right) \\
    & \geq & \frac{1}{3M}\left(\frac{1}{4} - \left( \frac{1}{3\sqrt{B}} + \frac{1}{5S}\right)2(1+\tau)^2\right).
\end{eqnarray*}
Finally, setting $\tau=\frac{1}{2}N^{1/5}-1$, $B=4N^{4/5}$ and $S=N^{2/5}$ yields $\gamma\geq\frac{1}{15M}$.
\end{proof}  
%Notice that $c_0$ is even smaller by using $|\Bkz|=N$ and $|\Skz|=N$.

Combining Lemma~\ref{lemma:Lyapunov-finite-nonsmooth} and Lemma~\ref{lemma:gamma-nonsmooth}, we arrive at the main result of this section.
\begin{theorem}
	\label{theorem:svrg}
    Suppose Assumptions~\ref{assumption:f-cvx}, \ref{assumption:phi-lowerbound}  and~\ref{assumption:g-lip-finite} hold for problem~\eqref{eqn:composite-finite}. Let the estimates $\tgkz$, $\tJkz$, $\tgki$ and $\tJki$ in Algorithm~\ref{algo:SVR-PL} be given in~\eqref{eqn:finite-gk0}-\eqref{eqn:J-svrg} respectively. If we choose $M\geq 4\ell_f L_g$ and $\tau = \bigl\lceil\frac{1}{2}N^{1/5}-1\bigr\rceil$, and
\[
    |\cB_i^k| = \bigr\lceil 4N^{4/5}\bigr\rceil, \qquad
    |\cS_i^k| = \bigr\lceil N^{2/5}\bigr\rceil, \qquad 
    i=1,\ldots,\tau-1,  \qquad k=1,\ldots,K,
\]
then the output of Algorithm~\ref{algo:SVR-PL} satisfies
\begin{equation}\label{eqn:cvg-svr-pl-nonsmooth-1}
    \E\bigl[\|\cG(x_{i^*}^{k^*})\|^2\bigr] ~\leq~ \frac{15M\bigl(\Phi(x_0^1) - \Phi_*\bigr)}{K\tau}.
\end{equation}
To get an $\epsilon$-stationary point in expectation, the total sample complexity for the component mappings $g_j$ and their Jacobians are both $\cO(N+N^{4/5}\epsilon^{-1})$.
\end{theorem}
\begin{proof}
	Summing up the inequality \eqref{lm:Lyapunov-finite-nonsmooth-1} 
    over~$k$ from~$1$ to~$K$ and using the fact $\Phi(x^K_\tau)>\Phi_*$, we get
	\begin{eqnarray*}
		\sum_{k=1}^K\sum_{i=0}^{\tau-1}\E\bigl[\|\cGki\|^2\bigr] ~\leq~ \frac{\Phi(x_0^1) - \Phi(x^*)}{\gamma}.
	\end{eqnarray*}
	By the random choice of the output $x_{i^*}^{k^*}$, we can get the inequality \eqref{eqn:cvg-svr-pl-nonsmooth-1}.
	
    To get an $\epsilon$-stationary point in expectation, we need to set $K\tau = \cO(\epsilon^{-1})$, which implies 
\[
    K=\cO(\tau^{-1}\epsilon^{-1})=\cO(N^{-1/5}\epsilon^{-1}).
\] 
Consequently, the sample complexity of the component mappings (the $g_i$'s) is 
\[
    KN + K\tau B 
    ~=~ \cO(N^{-1/5}\epsilon^{-1})\cdot N + \cO(\epsilon^{-1}) \cdot 4 N^{4/5} 
    ~=~ \cO(N+N^{4/5}\epsilon^{-1}).
\]
and the sample complexity for the component Jacobians is
\[
    KN + K\tau S 
    ~=~ \cO(N^{-1/5}\epsilon^{-1})\cdot N + \cO(\epsilon^{-1}) \cdot N^{2/5} 
    ~=~ \cO(N+N^{4/5}\epsilon^{-1}).
\]
This completes the proof.
\end{proof}

\begin{remark}\label{remark:dependent-batches}
		 Up to this point, we notice that all the analysis leading to Theorem \ref{theorem:svrg} only requires the sample batches between iterations to be independent. Whereas within each iteration we do not require the independence between $\cB_i^k$ and $\cS_i^k$. Therefore, in practice one can simply use the same mini-batch $\cB_i^k = \cS_i^k$ to estimate both $\tilde g_i^k$ and $\tilde J_i^k$, with batch size equal to $\max\{S,B\} = \lceil4N^{4/5}\rceil$.
Or, we can use a random subset of~$\mathcal{B}_i^k$ of size $\lceil N^{2/5}\rceil$ to compute $\tilde J_i^k$ in order to save computation.
\end{remark}

\begin{remark}\label{remark:nonsmooth-finite}
The nonsmooth and finite-sum case is also considered in \cite{tran2020stochastic}. But their results are limited to using the simple mini-batch scheme for both
component mapping and Jacobian estimation.
As a consequence, their sample complexities for the component mappings and their Jacobians are $\cO(\epsilon^{-3})$ and $\cO(\epsilon^{-2})$ respectively, without explicit dependence on~$N$.
They are similar to our results in Section~\ref{sec:mini-batch} on using mini-batches when~$g$ is a general expectation.
\end{remark}

\section{The nonsmooth and expectation case}
\label{sec:nonsmooth-expect}

In this section, we consider the composite stochastic optimization 
problem~\eqref{eqn:composite-expect}, which we repeat here for convenience:
\[
    \minimize_x ~~\Phi(x)\triangleq f(g(x)) + h(x), \quad 
    \mbox{where}\quad g(x) = \E_\xi \bigl[ g_\xi(x) \bigr].
\]
We assume that~$f$ and~$h$ satisfy Assumption~\ref{assumption:f-cvx} and 
the $g_\xi$'s satisfy the following assumption.
%which implies that their expectation~$g$ satisfies Assumption~\ref{assumption:g-lip}.
\begin{assumption}
	\label{assumption:g-lip-infinite}
% For almost all $\xi$, the mapping $g_\xi:\R^n\to\R^m$ is $\ell_g$-Lipschitz continuous and its Jacobian matrix $g'_\xi$ is $L_g$-Lipschitz continuous.  
The random mappings $g_\xi:\R^n\to\R^m$ and their Jacobians are mean-squares Lipschitz continuous, i.e., there exist constants $\ell_g$ and $L_g$ such that for all $x, y\in\dom h$, 
\begin{eqnarray*}
    \E\bigl[\|g_\xi(x) - g_\xi(y)\|^2\bigr] &\leq& \ell_g^2\|x-y\|^2,\\
    \E\bigl[\|g'_\xi(x) - g'_\xi(y)\|^2\bigr] &\leq& L_g^2\|x-y\|^2.
\end{eqnarray*}
    Furthermore, there exist constants $\sigma_g^2$ and $\sigma_{g'}^2$ such that for all $x\in\dom h$,
\begin{eqnarray*}
    \E\bigl[\|g_\xi(x) - g(x)\|^2\bigr] &\leq& \sigma_g^2,\\
    \E\bigl[\|g'_\xi(x) - g'(x)\|^2\bigr] &\leq& \sigma_{g'}^2.
\end{eqnarray*}
\end{assumption}
Assumption~\ref{assumption:g-lip-infinite} implies Assumption~\ref{assumption:g-lip}, but is weaker than assuming that $g_\xi$ and $g'_\xi$ are almost surely $\ell_g$- and $L_g$-Lipschitz respectively.

In this case, the first-order correction used in~\eqref{eqn:g-svrg-correction} is no longer useful in reducing the estimation errors because we cannot evaluate $g(\xkz)$ or $g'(\xkz)$ accurately. 
Instead, we turn to the SARAH/\textsc{Spider} estimator developed in 
\cite{SARAH-1,fang2018spider}.
But before doing that, we first examine the simple mini-batch scheme outlined in~\eqref{eqn:mini-batch} and~\eqref{eqn:stoch-prox-linear}.

\subsection{The simple mini-batch method}
\label{sec:mini-batch}

\begin{algorithm2e}[t]
    \caption{Simple mini-batch prox-linear algorithm}
	\label{algo:mini-batch-PL} 
	{\bf input:} initial point $x_0$, parameter $M>0$, and number of iterations $T$.\\
	\For{$i = 0,\ldots,T-1$}{ 
        sample mini-batches $\cB_i$ and $\cS_i$ from distribution of~$\xi$, 
        and compute $\tgi$ and $\tJi$ as in~\eqref{eqn:Lrg-batch-infinite}.\\
        $x_{i+1} = \displaystyle\argmin_x ~\textstyle\Bigl\{f\bigl(\tgi + \tJi(x-x_i)\bigr) + h(x) + \frac{M}{2}\|x-x_i\|^2\Bigr\}$. 
	}  
    {\bf output:} choose $x_{i^*}$ from $\{x_0,x_1,\ldots,x_{T-1}\}$ uniformaly at random.
\end{algorithm2e}   

The simple mini-batch method is to run Algorithm~\ref{algo:SVR-PL} with only one epoch ($K=1$) and $\tau=T$ iterations, where during each iteration we set
\begin{equation}
\label{eqn:Lrg-batch-infinite}
\tgi =  \frac{1}{|\cB_i|}\sum_{\xi \in\cB_i} g_\xi(x_i), \qquad\mbox{and} \qquad
\tJi = \frac{1}{|\cS_i|}\sum_{\xi \in\cS_i} g'_{\xi}(x_i).
\end{equation}
Since there is only one epoch, we omit the superscript~$k$ on $\xki$, $\tgki$ and $\tJki$ to write $x_i$, $\tgi$ and $\tJi$. Similar to Remark \ref{remark:dependent-batches}, we do not require the independence between $\cB_i$ and $\cS_i$. For clarity, we present the resulting method as Algorithm~\ref{algo:mini-batch-PL}.
The following complexity result holds.
\begin{theorem}\label{thm:nonsmooth-infinite-mini-batch}
    Suppose Assumptions~\ref{assumption:f-cvx}, \ref{assumption:phi-lowerbound} and~\ref{assumption:g-lip-infinite} hold for problem~\eqref{eqn:composite-expect}.
    If we choose $M\geq 4\ell_f L_g$ and the batch sizes  
    $|\cB_i| = B \geq \frac{36\ell_f^2\sigma_g^2}{\epsilon^2}$ and 
    $|\cS_i| = S \geq \frac{2\ell_f\sigma_{g'}^2}{L_g\epsilon}$, 
    then the output $x_{i^*}$ of Algorithm~\ref{algo:mini-batch-PL} satisfies 
\begin{equation}\label{eqn:thm-nonsmooth-infinite}
    \E\bigl[\|\cG(x_{i^*})\|^2\bigr]\leq 12 M \left(\frac{\Phi(x_0)-\Phi_*}{T} + \epsilon\right).
\end{equation}
Consequently by setting $T = \cO(\epsilon^{-1})$, the sample complexities for the component mappings $g_\xi$ and their Jacobians for getting an $\epsilon$-solution are $\cO(\epsilon^{-3})$ and $\cO(\epsilon^{-2})$ respectively.
\end{theorem}
\begin{proof}
    From the construction of $\tgi$ and $\tJi$ in~\eqref{eqn:Lrg-batch-infinite}, we have $\E[\tgi]=g(x_i)$ and $\E[\tJi]=g'(x_i)$. Moreover, by Assumption~\ref{assumption:g-lip-infinite}, we have
\[
\E\bigl[\|\tgi-g(x_i)\|^2\bigr]\leq \frac{\sigma_g^2}{B}, \qquad
\E\bigl[\|\tJi-g'(x_i)\|^2\bigr]\leq \frac{\sigma_{g'}^2}{S}.
\] 
Using Jensen's inequality, the variance bound on $\tgi$ further implies that $\E\bigl[\|\tgi-g(x_i)\|\bigr]\leq \frac{\sigma_g}{\sqrt{B}}$. 
Together with Lemma~\ref{lemma: descent-Finite-nonsmooth},  we have 
\begin{equation}\label{eqn:tcG-Phi}
	\frac{M-2\ell_f L_g}{2M^2}\E\bigl[\|\tcGki\|^2\bigr] ~\leq~ \E[\Phi(x_i)] - \E[\Phi(x_{i+1})]  + \frac{2\ell_f\sigma_g}{\sqrt{B}} + \frac{\ell_f\sigma_{g'}^2}{2L_gS}. 
\end{equation}
On the other hand, applying Lemma~\ref{lemma:PG-finite-sum-nonsmooth} yields
\begin{equation}\label{eqn:cG-tcG}
\frac{M-\ell_f L_g}{M^2}\E\bigl[\|\cGki\|^2\bigr]
~\leq~ \frac{2M+\ell_f L_g}{M^2}\E\bigl[\|\tcGki\|^2\bigr] 
+ \frac{4\ell_f\sigma_g}{\sqrt{B}} + \frac{2\ell_f\sigma_{g'}^2}{L_gS}.
\end{equation}
Next, we multiply both sides of~\eqref{eqn:cG-tcG} by $\frac{M-2\ell_f L_g}{2(2M+\ell_f L_g)}$ and add them to~\eqref{eqn:tcG-Phi} to cancel the terms containing $\E\bigl[\|\tcGki\|^2\bigr]$.
Then with $M\geq 4\ell_f L_g$, we have $\frac{M-2\ell_f L_g}{2(2M+\ell_f L_g)}\in\bigl[\frac{1}{9},\frac{1}{4}\bigr]$ and obtain 
\[
\frac{1}{12M}\E[\|\cG(x_i)\|^2] \leq \E[\Phi(x_i)] - \E[\Phi(x_{i+1})]  + \frac{3\ell_f\sigma_g}{\sqrt{B}}+ \frac{\ell_f\sigma_{g'}^2}{L_gS}.
\]
Summing up the above inequality over~$i$ from~$0$ to~$T-1$ and dividing by~$T$, we obtain
\[
    \frac{1}{T}\sum_{i=0}^{T-1}\E\bigl[\|\cG(x_i)\|^2\bigr] 
    ~\leq~ 12M\left(\frac{\Phi(x_0) - \Phi_*}{T}  + \frac{3\ell_f\sigma_g}{\sqrt{B}}+ \frac{\ell_f\sigma_{g'}^2}{L_gS}\right).
\]
Finally, using $|\cB_i| = B \geq \frac{36\ell_f^2\sigma_g^2}{\epsilon^2}$ and 
$|\cS_i| = S \geq \frac{2\ell_f\sigma_{g'}^2}{L_g\epsilon}$ 
yields~\eqref{eqn:thm-nonsmooth-infinite}.
The sample complexities for $g_\xi$ and $g'_\xi$ can be obtained as $TB=\cO(\epsilon^{-3})$ and $TS=\cO(\epsilon^{-2})$ respectively.
\end{proof}

\subsection{Using the SARAH/SPIDER estimator}

In this section, we show that by using the SARAH/\textsc{Spider} estimator \cite{SARAH-1,fang2018spider}, the sample complexities for the component mappings and Jacobians can be improved to $\cO(\epsilon^{-5/2})$ and $\cO(\epsilon^{-3/2})$, respectively. 
We note that for solving problem~\eqref{eqn:composite-expect} when~$f$ is nonsmooth and convex (more generally weakly convex),
even the $\cO(\epsilon^{-3})$ and $\cO(\epsilon^{-2})$ sample complexities established in Theorem~\ref{thm:nonsmooth-infinite-mini-batch} seem to be new in the literature.

The SARAH/\textsc{Spider} estimators for Algorithm~\ref{algo:SVR-PL} are constructed as follows. 
For $i=0$, we set
\begin{equation}
\label{eqn:g-J-spider-0}
\tgkz =  \frac{1}{|\Bkz|}\sum_{\xi \in\Bkz} g_\xi(\xkz), \qquad\mbox{and} \qquad \tJkz = \frac{1}{|\Skz|}\sum_{j \in\Skz} g'_{\xi}(\xkz).
\end{equation}
For the rest iterations with $i=1,\ldots,\tau-1$,
\begin{eqnarray}
    \tgki &=& \tgkim + \frac{1}{|\cB_i^k|}\sum_{\xi \in\cB_i^k} \bigl(g_\xi(\xki) - g_\xi(\xkim)\bigr),  \label{eqn:g-spider-infinite}\\
    \tJki &=& \tJkim + \frac{1}{|\cS_i^k|}\sum_{\xi \in\cS_i^k} \bigl(g'_\xi(\xki) - g'_\xi(\xkim)\bigr), \label{eqn:J-spider-infinite} 
\end{eqnarray}
Here $\Bki$ and $\Ski$ for $i=0,1,\ldots,\tau-1$ are  mini-batches sampled from the underlying distribution of the random variable~$\xi$. We require the batches $\Bki$ (and $\Ski$) to be independently sampled for different iterations, whereas in each iteration $\Bki$ and $\Ski$ can be dependent or even identical. 
The mean-squared estimation errors of the above estimators are bounded via the following lemma, which is adapted from \cite[Lemma~2]{SARAH-1} or \cite[Lemma~1]{fang2018spider}. 
A complete proof can be found in \cite[Lemma~1]{C-SARAH}.
\begin{lemma}
	\label{lemma:sarah-mse-infinite}
    Suppose Assumption~\ref{assumption:g-lip-infinite} holds and 
$\tgki$ and $\tJki$ are constructed through \eqref{eqn:g-J-spider-0}-\eqref{eqn:J-spider-infinite}. Then we have for $k=1,\ldots,K$ and $\tau=0,1,\ldots,\tau-1$,
\begin{eqnarray}
\!\!\!\!\!\!\!\!    \E\bigl[\|\tgki-g(\xki)\|^2\bigr] &\leq& \E\bigl[\|\tgkz-g(\xkz)\|^2\bigr] + \sum_{r=1}^{i}\frac{\ell_g^2}{|\cB^k_r|}\E\bigl[\|x_r^k-x_{r-1}^k\|^2\bigr],
    \label{eqn:g-spider-mse}\\
\!\!\!\!\!\!\!\!    \E\bigl[\|\tJki-g'(\xki)\|^2\bigr] &\leq&  \E\bigl[\|\tJkz-g'(\xkz)\|^2\bigr] + \sum_{r=1}^{i}\frac{L_g^2}{|\cS_r^k|}\E\bigl[\|x_r^k-x_{r-1}^k\|^2\bigr].
    \label{eqn:J-spider-mse}
\end{eqnarray}
\end{lemma}

The following theorem establishes the convergence of Algorithm~\ref{algo:SVR-PL} by specifying the batch sizes used in the SARAH/\textsc{Spider} estimators, and gives the sample complexities for $g_\xi$ and $g'_\xi$.
\begin{theorem}
    Suppose Assumptions~\ref{assumption:f-cvx}, \ref{assumption:phi-lowerbound} and~\ref{assumption:g-lip-infinite} hold for problem~\eqref{eqn:composite-expect}.
    Let the estimates $\tgkz$, $\tJkz$, $\tgki$ and $\tJki$ in Algorithm~\ref{algo:SVR-PL} be given in~\eqref{eqn:g-J-spider-0}-\eqref{eqn:J-spider-infinite}.
    If we choose $M\geq 4\ell_f L_g$ and $\tau=\lceil\epsilon^{-1/2}\rceil$, and
    the batch sizes as 
\[
    |\Bkz| = \biggl\lceil\frac{25\ell_f^2\sigma_g^2}{4\epsilon^2}\biggr\rceil, \quad
    |\Skz| = \biggl\lceil\frac{3\ell_f\sigma_{g'}^2}{4L_g\epsilon}\biggr\rceil, \quad
    |\Bki| = \biggl\lceil\frac{25\ell_f^2\ell_g^2}{M\epsilon^{3/2}}\biggr\rceil,  \quad
    |\Ski| = \left\lceil\frac{12\ell_f L_g}{M\epsilon^{1/2}}\right\rceil, 
\] 
for $i=1,\ldots,\tau-1$, 
then the output $x_{i^*}^{k^*}$ satisfies 
	\begin{equation}
	\label{eqn:cvg-sarah-nonsmooth-infinite}
	\E\Bigl[\bigl\|\cG(x_{i*}^{k^*})\bigr\|^2\Bigr]~\leq~ 24M\left( \frac{\Phi(x_0^1)-\Phi_*}{K\tau}+3\epsilon\right).
	\end{equation}
Consequently by setting $K = \cO(\epsilon^{-\half})$, then we get an output $\E[\|\cG(x^{k^*}_{i^*})\|^2]\leq \cO(\epsilon)$ with a function evaluation complexity of $\cO(\epsilon^{-5/2})$ and a Jacobian evaluation complexity of $\cO(\epsilon^{-3/2})$.
\end{theorem}
\begin{proof}
    We will choose batch sizes that do not depend on~$k$.
	For the ease of notation, we set $|\Bkz|=B$, $|\Skz|=S$, and $|\Bki|= b$ and $|\Ski|=s$ for $i=1,\ldots,\tau-1$. 
    First, by Assumption~\ref{assumption:g-lip-infinite} and~\eqref{eqn:g-J-spider-0}, we have
\[
    \E\bigl[\|\tgkz-g(\xkz)\|^2\bigr] = \frac{\sigma_g^2}{B}, \qquad
    \E\bigl[\|\tJkz-g'(\xkz)\|^2\bigr] = \frac{\sigma_{g'}^2}{S},
\]
which can be substituted into Lemma~\ref{lemma:sarah-mse-infinite}.
Then by Lemma~\ref{lemma:sarah-mse-infinite}, we know that 
	$$\E\bigl[\|\tgki-g(\xki)\|\bigr] ~\leq~ \sqrt{\E\bigl[\|\tgki-g(\xki)\|^2\bigr]} ~\leq~ \frac{\sigma_g}{\sqrt{B}} + \sqrt{\frac{\ell_g^2}{b}\sum_{r=1}^{i}\E\bigl[\|x_r^k-x_{r-1}^k\|^2\bigr]}.$$
Moreover, for any $\delta>0$, we have 
    $$\sqrt{\frac{\ell_g^2}{b}\sum_{r=1}^{i}\E\bigl[\|x_r^k-x_{r-1}^k\|^2\bigr]}~\leq~ \frac{\delta}{2} + \frac{\ell_g^2}{2b\delta}\sum_{r=1}^{i}\E\bigl[\|x_r^k-x_{r-1}^k\|^2\bigr].$$
Now we invoke Lemma~\ref{lemma: descent-Finite-nonsmooth}. Taking expectation on both sides of~\eqref{eqn:descent-lemma} and applying~\eqref{eqn:J-spider-mse} and the above bounds, we obtain
\begin{eqnarray} 
	\E\bigl[\Phi(\xkie)\bigr]  
    & \leq & \E[\Phi(\xki)] - \frac{M-2\ell_f L_g}{2}\E\bigl[\|\xkie-\xki\|^2\bigr] +  \frac{\ell_f\sigma_{g'}^2}{2L_gS} + \frac{2\ell_f\sigma_g}{\sqrt{B}}\nonumber \\
    && +\, \biggl(\frac{\ell_f L_g}{2s} + \frac{\ell_f\ell_g^2}{b\delta}\biggr)\sum_{r=1}^i\E\bigl[\|x_r^k-x_{r-1}^{k}\|^2\bigr] + \ell_f\delta ,
	\label{eqn:sarah-nonsmooth-1}
\end{eqnarray}
Similarly, with Lemma~\ref{lemma:PG-finite-sum-nonsmooth}, we have
\begin{eqnarray}
    \frac{M - \ell_f L_g}{M^2}\E\bigl[\|\cGki\|^2\bigr]
&\leq& (2M+\ell_f L_g) \E\bigl[\|\xkie-\xki\|^2\bigr] + \frac{2\ell_f\sigma_{g'}^2}{L_gS} + \frac{4\ell_f\sigma_g}{\sqrt{B}} \nonumber\\
    && + \biggl(\frac{2\ell_f L_g}{s} + \frac{2\ell_f\ell_g^2}{b\delta}\biggr)\sum_{r=1}^i\E\bigl[\|x_r^k-x_{r-1}^{k}\|^2\bigr] + 2\ell_f\delta.
	\label{eqn:sarah-nonsmooth-2}
\end{eqnarray}
Because $M \geq 4\ell_f L_g$, we have $\frac{1}{2}\cdot\frac{M-2\ell_f L_g}{2(2M + \ell_f L_g)}\in\bigl[\frac{1}{18}, \frac{1}{8}\bigr]$. Therefore, multiplying\eqref{eqn:sarah-nonsmooth-2} by $\frac{1}{2}\cdot\frac{M-2\ell_f L_g}{2(2M + \ell_f L_g)}$ and adding to~\eqref{eqn:sarah-nonsmooth-1} gives
\begin{eqnarray*}
	\frac{1}{24M}\E\bigl[\|\cGki\|^2\bigr] 
    & \leq & \E\bigl[\Phi(\xki)\bigr] - \E\bigl[\Phi(\xkie)\bigr] - \frac{M-2\ell_f L_g}{4}\E\bigl[\|\xkie-\xki\|^2\bigr] \\
	& & +\, \biggl(\frac{3 \ell_f L_g}{4s} + \frac{5\ell_f\ell_g^2}{4b\delta}\biggr)\sum_{r=1}^i\E\bigl[\|x_r^k-x_{r-1}^{k}\|^2\bigr] 
     + \frac{3\ell_f\sigma_{g'}^2}{4L_gS} + \frac{5\ell_f\sigma_g}{2\sqrt{B}} + \frac{5}{4}\ell_f\delta.
\end{eqnarray*}
Next, we replace $\sum_{r=1}^i\E\bigl[\|x_r^k-x_{r-1}^{k}\|^2\bigr]$ in the above inequality by $\sum_{r=1}^\tau\E\bigl[\|x_r^k-x_{r-1}^{k}\|^2\bigr]$. 
Then summing up the above inequality for $i = 0,\ldots,\tau-1$ gives
\begin{eqnarray*}
	\frac{1}{24M}\sum_{i=0}^{\tau-1}\E\bigl[\|\cGki\|^2\bigr] 
    \!&\leq&\! \E[\Phi(x^k_0)] - \E[\Phi(x^k_\tau)] \\
      && - \biggl(\frac{M}{8} - \frac{3\tau \ell_f L_g}{4s} - \frac{5\tau\ell_f\ell_g^2}{4b\delta}\biggr)\sum_{r=1}^{\tau}\E\bigl[\|x_r^k-x_{r-1}^k\|^2\bigr]\\
	& &  +\, \biggl(\frac{3\ell_f\sigma_{g'}^2}{4L_gS} + \frac{5\ell_f\sigma_g}{2\sqrt{B}}+\frac{5}{4}\ell_f\delta\biggr)\tau.
\end{eqnarray*}
	If we set $\delta = \frac{4\epsilon}{5\ell_f}$, 
    $B= \frac{25\ell_f^2\sigma_g^2}{4\epsilon^2}$, 
    $S = \frac{3\ell_f\sigma_{g'}^2}{4L_g\epsilon}$, 
    $s = \frac{12\ell_f L_g}{M}\tau$ and
    $b = \frac{20\ell_f\ell_g^2}{M\delta}\tau = \frac{25\ell_f^2\ell_g^2}{M\epsilon}\tau$, then 
	$$\frac{M}{8} - \frac{3\tau \ell_f L_g}{4s} - \frac{5\tau\ell_g^2\ell_f}{4b\delta}~\geq~ 0 \quad\mbox{ and }\quad
    \frac{3\ell_f\sigma_{g'}^2}{4L_gS} + \frac{5\ell_f\sigma_g}{2\sqrt{B}}+\frac{4}{5}\ell_f\delta ~\leq~ 3\epsilon.$$
Therefore,
	\begin{eqnarray*}
		\frac{1}{24M}\sum_{i=0}^{\tau-1}\E\bigl[\|\cGki\|^2\bigr] 
        & \leq & \E[\Phi(x^k_0)] - \E[\Phi(x^k_\tau)] + 3\tau\epsilon.
	\end{eqnarray*}
    Summing up the above inequality for $k=1,\ldots,K$, and dividing by~$K\tau$, we obtain
	\begin{eqnarray*}
        \frac{1}{K\tau}\sum_{k=1}^K\sum_{i=0}^{\tau-1}\E\bigl[\|\cGki\|^2\bigr] 
        & \leq & 24M\left(\frac{\Phi(x^1_0) - \Phi_*}{K\tau} + 3\epsilon \right).
	\end{eqnarray*}
Since $x^{k^*}_{i^*}$ is randomly chosen from $\bigl\{\xki\bigr\}_{i=0,\ldots,\tau-1}^{k=1,\ldots,K}$, it satisfies \eqref{eqn:cvg-sarah-nonsmooth-infinite}. Moreover, in this case, we have $b=\cO(\tau/\epsilon)$ and $s=\cO(\tau)$. To find an $\epsilon$-stationary point in expectation, we further set $\tau = \epsilon^{-1/2}$ and $K = \cO(\epsilon^{-1/2})$, which implies $\E\bigl[\|\cG(x^{k^*}_{i^*})\|^2\bigr]\leq \cO(\epsilon)$.
Consequently, the sample complexity for the component mappings is
\[
    KB + K\tau b ~=~ \cO(\epsilon^{-1/2})\cdot\cO(\epsilon^{-2})
    + \cO(\epsilon^{-1/2})\cdot\epsilon^{-1/2}\cdot\cO(\epsilon^{-3/2})
    ~=~ \cO(\epsilon^{-5/2}),
\]
and the sample complexity for the Jacobians is
\[
    KS + K\tau s ~=~ \cO(\epsilon^{-1/2})\cdot\cO(\epsilon^{-1})
    + \cO(\epsilon^{-1/2})\cdot\epsilon^{-1/2}\cdot\cO(\epsilon^{-1/2})
    ~=~ \cO(\epsilon^{-3/2}).
\]
This finishes the proof.
\end{proof}

\begin{remark}\label{remark:connection-to-SGN2}
The sample complexities $\cO(\epsilon^{-5/2})$ and $\cO(\epsilon^{-3/2})$,
for component mappings and their Jacobians respectively, 
are also obtained using the SARAH estimator in \cite{tran2020stochastic}, 
but for a slightly different stationarity measure.
Specifically, their results are derived for finding a point $x$ that satisfies 
$\E[\|\widetilde{\mathcal{G}}(x)\|^2]\leq \epsilon$, 
where $\widetilde{\mathcal{G}}(x)$ is the \emph{approximate} gradient mapping
defined in~\eqref{eqn:xkie-def} and~\eqref{eqn:approx-gm}.
Since $\|\widetilde{\mathcal{G}}(x)\|=0$ alone may
not be a good measure for stationarity, 
\cite{tran2020stochastic} defined a primal-dual stationarity measure 
which requires additional conditions.
In contrast, our results directly guarantee
 $\E[\|\mathcal{G}(x)\|^2]\leq \epsilon$, 
where $\mathcal{G}(x)$ is the (exact) gradient mapping defined in~\eqref{eqn:grad-mapping}.
\end{remark}

\section{The smooth and finite-average case}
\label{sec:smooth-finite} 

In this section, we consider problem~\eqref{eqn:composite-finite} under the assumption that~$f$ is smooth and convex. 
Specifically, we assume that the component mappings $g_i$ satisfy Assumption~\ref{assumption:g-lip-finite}. For~$f$ and~$h$, in addition to Assumption~\ref{assumption:f-cvx}, we make the following additional assumption.

\begin{assumption}
	\label{assumption:SVR-PL-f-smooth} The gradient of~$f$, denoted as $f'$, is differentiable and $L_f$-Lipschitz continuous. 
\end{assumption} 
Under Assumptions~\ref{assumption:g-lip-finite} and~\ref{assumption:SVR-PL-f-smooth}, the composite function $f\circ g$ is smooth and its gradient has a Lipschitz constant
\begin{equation}\label{eqn:L-f-g-def}
    L_{f\circ g} \triangleq \ell_f L_g + L_f\ell_g^2.
\end{equation}
See \cite{C-SARAH} for a proof of this claim.

Algorithms for solving problem~\eqref{eqn:composite-finite} under the above assumptions have been studied in \cite{VRSC-PG,SVR-SCGD,C-SAGA,C-SARAH}.
The best sample complexity is $\cO(N+N^{1/2}\epsilon^{-1})$ obtained in \cite{C-SARAH}, using the SARAH/\textsc{Spider} estimator for $g'(x)f'(g(x))$, which is the gradient of $f(g(x))$ by the chain rule. 
In this section, we study an algorithm using the proximal mapping of~$f$ instead of the composite gradient.
It is no surprising that we can attain the sample complexity here.
Despite the same sample complexity in theory,
it is often observed in practice that algorithms based on proximal mappings can be more efficient than those based on gradients (e.g., \cite{AsiDuchi2019Truncated,AsiDuchi2019StoProxPoint,DamekDima2019ModelBased}).
Therefore, it is very meaningful to establish the sample complexity of proximal-mapping based methods when~$f$ is smooth.

We again apply the SARAH/\textsc{Spider} estimator to construct $\tgki$ and $\tJki$. For $i>0$, we use~\eqref{eqn:g-spider-infinite} and~\eqref{eqn:J-spider-infinite}, where~$\xi$ is interpreted as a random index drawn from $\{1,\ldots,N\}$ with replacement.
For $i=0$, we exploit the finite-average structure of $g$ by using the construction in~\eqref{eqn:finite-gk0} and~\eqref{eqn:finite-Jk0}, i.e.,
\begin{equation}
\label{def:sarah-f&g-0}
\tgkz = g(\xkz) \qquad\mbox{ and } \qquad\tJkz = g'(\xkz).
\end{equation}
This implies that 
$\E[\|\tgkz - g(x_0^k)\|^2]=0$ and $= \E[\|\tJkz - g'(x_0^k)\|^2]=0$, 
which can be substituted into Lemma \ref{lemma:sarah-mse-infinite}
to get the following result.
\begin{corollary}
	\label{lemma:sarah-mse-finite}
    Suppose Assumption~\ref{assumption:g-lip-infinite} holds.
    Let $\tgki$ and $\tJki$ be constructed according to~\eqref{def:sarah-f&g-0} for $i=0$ and~\eqref{eqn:g-spider-infinite} and~\eqref{eqn:J-spider-infinite} for $i=1,\ldots,\tau-1$. Then we have for $i=0,1,\ldots,\tau-1$,
\begin{eqnarray*}
    \E\bigl[\|\tgki-g(\xki)\|^2\bigr] &\leq& \sum_{r=1}^{i}\frac{\ell_g^2}{|\cB_r^k|}\E\bigl[\|x_r^k-x_{r-1}^k\|^2\bigr],\\
    \E\bigl[\|\tJki-g'(\xki)\|^2\bigr] &\leq& \sum_{r=1}^{i}\frac{L_g^2}{|\cS_r^k|}\E\bigl[\|x_r^k-x_{r-1}^k\|^2\bigr].
\end{eqnarray*}
\end{corollary}

Next, we prove a descent property of the algorithm. 
The additional assumption that~$f$ is smooth allows us to derive a tighter descent bound than Lemma~\ref{lemma: descent-Finite-nonsmooth}. 
In particular, we can replace the term $2\ell_f \|\tgki-g(\xki)\|$ in~\eqref{eqn:descent-lemma} with $L_f\|\tgki-g(\xki)\|^2$, which leads to reduction of the sample complexity for the component mappings.
\begin{lemma}
	\label{lemma: descent-finite-sum}
    Suppose Assumptions~\ref{assumption:f-cvx}, \ref{assumption:g-lip} and~\ref{assumption:SVR-PL-f-smooth} hold. Then Algorithm~\ref{algo:SVR-PL} has the following descent property:
\begin{eqnarray}
\Phi(\xkie) 
&\leq& \Phi(\xki) - \frac{M-2L_{f\circ g}}{2M^2}\bigl\|\tcGki\bigr\|^2 
\nonumber \\
&&  +\, L_f\bigl\|\tgki - g(\xki)\bigr\|^2
	+ \frac{\ell_f}{2L_g}\bigl\|\tJki - g'(\xki)\bigr\|^2 , 
    \label{eqn:Phi-descent-smooth}
\end{eqnarray}
where $L_{f\circ g}$ is defined in~\eqref{eqn:L-f-g-def}.
\end{lemma}
\begin{proof}
We revisit the proof of Lemma~\ref{lemma: descent-Finite-nonsmooth}.
In particular, the inequality~\eqref{eqn:Phi-descent} still holds, i.e.,
\begin{equation}\label{eqn:Phi-descent-repeat}
    \Phi(\xkie) ~=~ \Phi(\xki)  - \frac{M-\ell_f L_g}{2} \|\xkie-\xki\|^2 + T_1+ T_2 +T_3. 
\end{equation}
Moreover, we can reuse the bound for $T_2$ in~\eqref{eqn:descent-lemma-T2}.
and only need to rebound the terms $T_1$ and $T_3$.

Under Assumption~\ref{assumption:SVR-PL-f-smooth}, we denote the Hessian of~$f$ as $f''$ and it holds that $\|f''(z)\|\leq L_f$ for all $z\in\R^m$.
For the ease of notation, we denote 
$$\Dki \triangleq \tgki - g(\xki), \qquad 
\zki\triangleq g(\xki) + g'(\xki)(\xkie-\xki).$$
Starting with $T_1$, which is defined in~\eqref{eqn:descent-finite-0},
we use the second-order Taylor expansion of~$f$ to obtain
\begin{eqnarray*}
	T_1 
    & = & f\bigl(\zki\bigr)-f\bigl(\zki+\Dki\bigr) \\
    & = & f\bigl(\zki\bigr) - \Bigl( f\bigl(\zki\bigr) + \bigl\langle f'\bigl(\zki\bigr) ,\Dki\bigr \rangle + \half(\Dki)^T f''\bigl(\zki + \theta\Dki\bigr)\Dki \Bigr) \\
    & = & - \bigl\langle f'\bigl(\zki\bigr) ,\Dki\bigr\rangle - \half(\Dki)^T f''\bigl(\zki + \theta\Dki\bigr)\Dki,
\end{eqnarray*}
where $\theta\in[0,1]$. 
Since $f$ is convex and the spectral norm of $f''$ is bounded by~$L_f$, we have
\begin{eqnarray*}
    T_1
    & \leq & - \bigl\langle f'\bigl(\zki\bigr) ,\Dki\bigr\rangle \\ 
	& \leq & - \bigl\langle f'\bigl(g(\xki)\bigr) ,\Dki \bigr\rangle + \bigl|\bigl\langle f'\bigl(g(\xki)\bigr) - f'\bigl(\zki)\bigr) ,\Dki \bigr\rangle\bigr|\\
    & \leq & - \bigl\langle f'\bigl(g(\xki)\bigr) ,\Dki \bigr\rangle  + L_f\bigl\|g(\xki)-\zki\bigr\|\|\Dki\| \\
	&=& - \bigl\langle f'\bigl(g(\xki)\bigr) ,\Dki \bigr\rangle  + L_f\bigl\|g'(\xki)(\xkie-\xki)\bigr\|\|\Dki\|.
\end{eqnarray*}
Notice that by Assumption \ref{assumption:g-lip-finite} we have $\|g'(\xki)\|\leq \ell_g$, which gives
\begin{eqnarray}
T_1 & \leq & - \bigl\langle f'\bigl(g(\xki)\bigr) ,\Dki \bigr\rangle + L_f\ell_g \|\xkie-\xki\| \|\Dki\| \nonumber \\
    & \leq & - \bigl\langle f'\bigl(g(\xki)\bigr) ,\Dki \bigr\rangle + L_f\biggl(\frac{\ell_g^2}{2}\|\xkie-\xki\|^2+\half\|\Dki\|^2 \biggr) \nonumber \\
    & = & \frac{L_f}{2}\bigl\|\Dki\bigr\|^2 + \frac{L_f\ell_g^2}{2}\|\xkie - \xki\|^2- \bigl\langle f'\bigl(g(\xki)\bigr) ,\Dki \bigr\rangle .
\label{eqn:smooth-T-1}
\end{eqnarray}
For the term $T_3$ in~\eqref{eqn:Phi-descent}, we have for some $\theta\in[0,1]$,
\begin{eqnarray*}
	T_3 & = & f\bigl(g(\xki) + \Dki\bigr)-f\bigl(g(\xki)\bigr)\\
	& = & f\bigl(g(\xki)\bigr) +\bigl\langle f'\bigl(g(\xki)\bigr),\Dki\bigr\rangle + \half(\Dki)^T f''\bigl(g(\xki) + \theta\Dki\bigr)\Dki - f\bigl(g(\xki)\bigr) \\ 
	& \leq & \frac{L_f}{2}\bigl\|\Dki\bigr\|^2+ \bigl\langle f'\bigl(g(\xki)\bigr),\Dki \bigr\rangle.
\end{eqnarray*}
Substituting the new bounds on~$T_1$ and~$T_3$ and the existing bound on $T_2$ in~\eqref{eqn:descent-lemma-T2} into~\eqref{eqn:Phi-descent-repeat}, we obtain 
\begin{eqnarray*}
\Phi(\xkie) 
&\leq& \Phi(\xki) - \left(\frac{M}{2}-\ell_f L_g - \frac{1}{2}L_f\ell_g^2\right)\bigl\|\xkie-\xki\bigr\|^2 \\
&& +\, L_f\bigl\|\tgki - g(\xki)\bigr\|^2 + \frac{\ell_f}{2L_g}\bigl\|\tJki - g'(\xki)\bigr\|^2 . 
\end{eqnarray*}
The desired result holds by noting the definitions of $L_{f\circ g}$ and $\tcGki$.
\end{proof}

%	It is worth noting that even though $\E[\langle f'(g(\xki)) ,\Dki \rangle]   = 0$ in the current finite sum settings, this may not be true in our algorithm for the general expecation case problem. However, the $+ \E[\langle f'(g(\xki)) ,\Dki \rangle]$ in $\E[T_3]$ and the $- \E[\langle f'(g(\xki)) ,\Dki \rangle]  $ in $\E[T_1]$ exactly cancel each other out regardless of the value of $\E[\langle f'(g(\xki)) ,\Dki \rangle]$.	

Parallel to Lemma \ref{lemma:PG-finite-sum-nonsmooth}, we have the following result.

\begin{lemma}
	\label{lemma:PG-finite-sum}
    Suppose Assumptions~\ref{assumption:f-cvx}, \ref{assumption:g-lip} and~\ref{assumption:SVR-PL-f-smooth} hold. 
    Let $\xkie$ and $\hxkie$ are defined in~\eqref{eqn:xkie-def} and~\eqref{eqn:x-hat-def} respectively.
    Then we have
\begin{eqnarray}
\frac{M - L_{f\circ g}}{M^2}\bigl\|\cGki\bigr\|^2
&\leq& \frac{2M+L_{f\circ g}}{M^2}\bigl\|\tcGki\bigr\|^2 
    \label{eqn:G-tG-smooth} \\
&& +\, 3L_f\bigl\|\tgki - g(\xki)\bigr\|^2
 + \frac{2\ell_f}{L_g}\bigl\|\tJki - g'(\xki)\bigr\|^2. \nonumber 
\end{eqnarray}
\end{lemma}
\begin{proof}
    We revisit the proof of Lemma~\ref{lemma:PG-finite-sum-nonsmooth}, and start with the inequality~\eqref{lm:PG-finite-sum-3}, which is
	$$M\|\hxkie - \xkie\|^2\leq F(\xkie;\xki)-\widetilde{F}(\xkie;\xki) + \widetilde{F}(\hat{x}^k_{i+1};\xki)-F(\hat{x}^k_{i+1};\xki).$$
We can establish a tighter bound for the right-hand-side when~$f$ is smooth.
From the definitions of $F$ and $\widetilde{F}$ in~\eqref{eqn:F-def} and~\eqref{eqn:tF-def} and the definitions of $T_1$ nd $T_2$ in~\eqref{eqn:descent-finite-0}, we have
\begin{eqnarray}
F(\xkie;\xki)-\widetilde{F}(\xkie;\xki) 
    &=& f\bigl(g(\xki)+g'(\xki)(\xkie-\xki)\bigr) - f\bigl(\tgki+\tJki(\xkie-\xki)\bigr) \nonumber \\
    & = & T_1+T_2 \nonumber \\
	& \leq &  \frac{L_f}{2}\bigl\|\tgki - g(\xki)\bigr\|^2 + \frac{\ell_f}{2L_g}\bigl\|\tJki - g'(\xki)\bigr\|^2  
\label{lm:PG-finite-sum-4} \\
	& & + \frac{\ell_f L_g  + L_f\ell_g^2}{2}\bigl\|\xkie - \xki\bigr\|^2 - \bigl\langle f'\bigl(g(\xki)\bigr) ,\Dki \bigr\rangle , \nonumber
\end{eqnarray}
where the last inequality is due to~\eqref{eqn:descent-lemma-T2} and~\eqref{eqn:smooth-T-1}.
Following similar arguments, we can derive 
\begin{eqnarray}
	\widetilde{F}(\hat{x}^k_{i+1};\xki)-F(\hat{x}^k_{i+1};\xki)
	& \leq &  L_f\bigl\|\tgki - g(\xki)\bigr\|^2 + \frac{\ell_f}{2L_g}\bigl\|\tJki - g'(\xki)\bigr\|^2 
\label{lm:PG-finite-sum-5} \\
	&&+ \frac{\ell_f L_g + L_f\ell_g^2}{2}\bigl\|\hxkie- \xki\bigr\|^2 + \bigl\langle f'\bigl(g(\xki)\bigr) ,\Dki \bigr\rangle.
    \nonumber
\end{eqnarray}
Summing up~\eqref{lm:PG-finite-sum-4} and~\eqref{lm:PG-finite-sum-5} and noting the definition of~$L_{f\circ g}$, we have
\begin{eqnarray*}
    M\bigl\|\hxkie-\xkie\bigr\|^2 
    &\leq& \frac{3L_f}{2}\bigl\|\tgki - g(\xki)\bigr\|^2 + \frac{\ell_f}{L_g}\bigl\|\tJki - g'(\xki)\bigr\|^2 \\
    && +\, \frac{L_{f\circ g}}{2}\bigl\|\hxkie- \xki\bigr\|^2 + \frac{L_{f\circ g}}{2}\bigl\|\xkie- \xki\bigr\|^2. 
\end{eqnarray*}
Combining the above inequality with 
$$M\bigl\|\hxkie-\xki\bigr\|^2
\leq 2M\bigl\|\xkie-\xki\bigr\|^2 + 2M\bigl\|\hxkie-\xkie\bigr\|^2$$
yields (following similar arguments at the end of proof for Lemma \ref{lemma:PG-finite-sum-nonsmooth})
\begin{eqnarray*}
(M - L_{f\circ g})\bigl\|\hxkie-\xki\bigr\|^2
&\leq& (2M+L_{f\circ g})\bigl\|\xkie-\xki\bigr\|^2 \\
&& +\, 3L_f\bigl\|\tgki - g(\xki)\bigr\|^2
   + \frac{2\ell_f}{L_g}\bigl\|\tJki - g'(\xki)\bigr\|^2.
\end{eqnarray*}
Finally we obtain the desired result using the definitions of
$\cG(\xki)$ and $\widetilde{\cG}(\xki)$.
\end{proof}

The main result of this section is given by the following theorem.
\begin{theorem}
	\label{theorem:sarah-cvg-finite}
    Suppose Assumptions~\ref{assumption:f-cvx}, \ref{assumption:g-lip}, \ref{assumption:phi-lowerbound} and~\ref{assumption:SVR-PL-f-smooth} hold for problem~\eqref{eqn:composite-finite}. 
    In Algorithm~\ref{algo:SVR-PL}, let $\tgki$ and $\tJki$ be constructed according to~\eqref{def:sarah-f&g-0} for $i=0$ and~\eqref{eqn:g-spider-infinite} and~\eqref{eqn:J-spider-infinite} for $i=1,\ldots,\tau-1$. 
    If we choose $M\geq 4 L_{f\circ g}$ and $\tau = \lceil\sqrt{N}\rceil$, and set the batch sizes $|\Bki| = |\Ski| = 2\lceil\sqrt{N}\rceil$ for $i=1,\ldots,\tau-1$, then
\begin{equation}
\label{eqn:sarah-cvg-finite-1}
\E\left[\bigl\|\cG(x_{i^*}^{k^*})\bigr\|^2\right] ~\leq~ \frac{24M\bigl(\Phi(x^1_0) - \Phi_*\bigr)}{K\tau}.
\end{equation} 
The total sample complexity of reaching an $\epsilon$-stationary point in expectation is $\cO(N+\sqrt{N}\epsilon^{-1})$.
\end{theorem}
\begin{proof}
    Under the assumption $M\geq 4 L_{f\circ g}$, we have $\frac{1}{2}\cdot\frac{M-2L_{f\circ g}}{2(2M+L_{f\circ g})}\in\bigl[\frac{1}{18}, \frac{1}{8}\bigr]$.
    Multiplying both sides of~\eqref{eqn:G-tG-smooth} by $\frac{1}{2}\cdot\frac{M-2L_{f\circ g}}{2(2M+L_{f\circ g})}$ and adding them to~\eqref{eqn:Phi-descent-smooth}, we obtain
\begin{eqnarray}
\frac{1}{24M}\bigl\|\cGki\bigr\|^2
&\leq& \Phi(\xki)-\Phi(\xkie) - \frac{M-2L_{f\circ g}}{4M^2}\bigl\|\tcGki\bigr\|^2 \nonumber \\
&& +\, \frac{11}{8}L_f\bigl\|\tgki - g(\xki)\bigr\|^2
	+ \frac{3\ell_f}{4L_g}\bigl\|\tJki - g'(\xki)\bigr\|^2.
\label{eqn:smooth-before-expect}
\end{eqnarray}
Taking expectation on both sides of the above inequality and applying 
Corollary~\ref{lemma:sarah-mse-finite}, we have 
\begin{eqnarray*}
\frac{1}{24M}\E\bigl[\bigl\|\cGki\bigr\|^2\bigr]
&\leq& \E\bigl[\Phi(\xki)\bigr]-\E\bigl[\Phi(\xkie)\bigr] - \frac{M-2L_{f\circ g}}{4M^2}\E\bigl[\bigl\|\tcGki\bigr\|^2\bigr] \\
&& +\, \frac{11}{8}L_f\ell_g^2 \sum_{r=1}^i\frac{1}{|\Bkr|}\E\bigl[\bigl\|x^k_r - x^k_{r-1}\bigr\|^2\bigr] \\
&& +\, \frac{3\ell_f L_g}{4}\sum_{r=1}^i \frac{1}{|\Skr|} \E\bigl[\bigl\|x^k_r - x^k_{r-1}\bigr\|^2\bigr].
\end{eqnarray*}
We will use constant batch sizes and let $|\Bki|=|\Ski|=S$ for all $k=1,\ldots,K$ and $i=1,\ldots,\tau-1$.
In addition, we can increase the summation from $\sum_{r=1}^i$ to $\sum_{r=1}^\tau$, which leads to
\begin{eqnarray*}
\frac{1}{24M}\E\bigl[\bigl\|\cGki\bigr\|^2\bigr]
&\leq& \E\bigl[\Phi(\xki)\bigr]-\E\bigl[\Phi(\xkie)\bigr] - \frac{M-2L_{f\circ g}}{4M^2}\E\bigl[\bigl\|\tcGki\bigr\|^2\bigr] \\
&& +\, \frac{11}{8}\frac{L_f\ell_g^2}{S} \sum_{r=1}^\tau \E\bigl[\bigl\|x^k_r - x^k_{r-1}\bigr\|^2\bigr] \\
&& +\, \frac{3\ell_f L_g}{4S}\sum_{r=1}^\tau \E\bigl[\bigl\|x^k_r - x^k_{r-1}\bigr\|^2\bigr].
\end{eqnarray*}
Plugging in $\tcGki=-M(\xkie-\xki)$ and noticing that
$$\frac{11}{8}\frac{L_f\ell_g^2}{S} + \frac{3\ell_f L_g}{4S} \leq \frac{3}{2}\frac{\ell_f L_g + L_f\ell_g^2}{S}=\frac{3}{2}\frac{L_{f\circ g}}{S},$$
%the above inequality implies
we obtain
\begin{eqnarray*}
\frac{1}{24M}\E\bigl[\bigl\|\cGki\bigr\|^2\bigr]
&\leq& \E\bigl[\Phi(\xki)\bigr]-\E\bigl[\Phi(\xkie)\bigr] - \frac{M-2L_{f\circ g}}{4}\E\bigl[\bigl\|\xkie-\xki\bigr\|^2\bigr] \\
&& +\, \frac{3}{2}\frac{L_{f\circ g}}{S} \sum_{r=1}^\tau \E\bigl[\bigl\|x^k_r - x^k_{r-1}\bigr\|^2\bigr].
\end{eqnarray*}
Summing up the above inequality for $i=0,\ldots,\tau-1$ yields
\begin{eqnarray*}
    \frac{1}{24M}\sum_{i=0}^{\tau-1}\E\bigl[\bigl\|\cGki\bigr\|^2\bigr]
    &\leq& \E\bigl[\Phi(\xkz)\bigr]-\E\bigl[\Phi(x^{k+1}_0)\bigr] 
    - \left(\frac{M}{4} - \frac{2\tau}{S}L_{f\circ g}\right) \sum_{i=0}^{\tau-1} \E\bigl[\bigl\|\xkie-\xki\bigr\|^2\bigr] .
\end{eqnarray*}
The choices of $M\geq 4 L_{f\circ g}$, $\tau=\lceil\sqrt{N}\rceil$ and $S=2\tau$ ensure $\frac{M}{4}-\frac{2\tau}{S}L_{f\circ g}\geq 0$.
Therefore
\[
    \frac{1}{24M}\sum_{i=0}^{\tau-1}\E\bigl[\bigl\|\cGki\bigr\|^2\bigr]
    \leq \E\bigl[\Phi(\xkz)\bigr]-\E\bigl[\Phi(x^{k+1}_0)\bigr].
\]
Summing up the above inequality for $k=1,\ldots,K$ and noticing the choice of $x^{k^*}_{i^*}$ in Algorithm~\ref{algo:SVR-PL}, we obtain~\eqref{eqn:sarah-cvg-finite-1}.

To get an $\epsilon$-stationary point in expectation, we need to set $K\tau = \cO(\epsilon^{-1})$, which implies 
\[
    K=\cO(\tau^{-1}\epsilon^{-1})=\cO(N^{-1/2}\epsilon^{-1}).
\] 
Consequently, the sample complexity for both the component mappings 
and their Jacobians is
\[
    KN + K\tau S 
    ~=~ \cO(N^{-1/2}\epsilon^{-1})\cdot N + \cO(\epsilon^{-1}) \cdot 2 N^{1/2} 
    ~=~ \cO(N+N^{1/2}\epsilon^{-1}).
\]
This finishes the proof.
\end{proof}

\section{The smooth and expectation case}
\label{sec:smooth-expect}

In this section we focus on problem~\eqref{eqn:composite-expect} when~$f$ is smooth and convex. Specifically, we proceed with
Assumptions~\ref{assumption:f-cvx}, \ref{assumption:g-lip-infinite} and \ref{assumption:SVR-PL-f-smooth}. 
Under these assumptions, we still use the SARAH/\textsc{Spider} estimators 
in~\eqref{eqn:g-J-spider-0}, \eqref{eqn:g-spider-infinite} and~\eqref{eqn:J-spider-infinite}.
Since that the mean-square error bounds bounds on the estimators in Lemma~\ref{lemma:sarah-mse-infinite} only depends on Assumption~\ref{assumption:g-lip-infinite}, they remain valid in this section. 
We have the following result.

\begin{theorem}
\label{theorem:SARAH-cvg-Infinite} 
Suppose Assumptions~\ref{assumption:f-cvx}, \ref{assumption:phi-lowerbound}, \ref{assumption:g-lip-infinite} and~\ref{assumption:SVR-PL-f-smooth} hold for problem~\eqref{eqn:composite-expect}.
    Let the estimates $\tgkz$, $\tJkz$, $\tgki$ and $\tJki$ in Algorithm~\ref{algo:SVR-PL} be given in~\eqref{eqn:g-J-spider-0}-\eqref{eqn:J-spider-infinite}, and we choose $M\geq 4 L_{f\circ g}$.
    For any $\epsilon>0$, if we choose $\tau = \lceil\epsilon^{-1/2}\rceil$ and the batch sizes as
\[
    \bigl|\Bkz\bigr|=\left\lceil\frac{11L_f\sigma_g^2}{4\epsilon}\right\rceil, \qquad
    \bigl|\Skz\bigr|=\left\lceil\frac{3\ell_f^2\sigma_{g'}^2}{2L_g\epsilon}\right\rceil, \qquad
    \bigl|\Bki\bigr|=\bigl|\Ski\bigr|=2 \left\lceil \epsilon^{-1/2}\right\rceil,% \qquad i=1,\ldots,\tau-1,
\]
where $i=1,\ldots,\tau-1$, 
then we have 
\begin{equation}
\label{thm:sarah-cvg-infinite-0}
\frac{1}{24M}\E[\|\cG(x_{i^*}^{k^*})\|^2] 
~\leq~ \frac{\Phi(x^1_0) - \Phi(x^*)}{K\tau} + \epsilon.
\end{equation}
Consequently, we have 
$\E\bigl[\|\cG(x_{i^*}^{k^*})\|^2\bigr] = \cO(\epsilon)$ 
by setting $K = \cO(\epsilon^{-1/2})$, and the total sample complexity is
$\cO(\epsilon^{-3/2})$.
\end{theorem}
\begin{proof}
    We choose batch sizes that do not depend on~$k$.
	For the ease of notation, let $|\Bkz|=B$ and $|\Skz|=S$, and $|\Bki| = |\Ski|=b$ for $i=1,\ldots,\tau-1$. 
    Taking expectation of both sizes of~\eqref{eqn:smooth-before-expect} and applying Lemma~\ref{lemma:sarah-mse-infinite}, we get
\begin{eqnarray*}
\frac{1}{24M}\E\bigl[\bigl\|\cGki\bigr\|^2\bigr]
&\leq& \E\bigl[\Phi(\xki)\bigr]-\E\bigl[\Phi(\xkie)\bigr] - \frac{M-2L_{f\circ g}}{4M^2}\E\bigl[\bigl\|\tcGki\bigr\|^2\bigr] \\
&& +\, \frac{11}{8}\frac{L_f\sigma_g^2}{B} + \frac{3\ell_f\sigma_{g'}^2}{4L_g S} + \frac{11}{8}\frac{L_f\ell_g^2}{b} \sum_{r=1}^i\E\bigl[\bigl\|x^k_r - x^k_{r-1}\bigr\|^2\bigr] \\
&& +\, \frac{3\ell_f L_g}{4b}\sum_{r=1}^i \E\bigl[\bigl\|x^k_r - x^k_{r-1}\bigr\|^2\bigr].
\end{eqnarray*}
Summing up the above inequality for $i=0,\ldots,\tau-1$ and following similar steps as in the proof of Theorem~\ref{theorem:sarah-cvg-finite}, we have
\begin{eqnarray*}
    \frac{1}{24M}\sum_{i=0}^{\tau-1}\E\bigl[\bigl\|\cGki\bigr\|^2\bigr]
    &\leq& \E\bigl[\Phi(\xkz)\bigr]-\E\bigl[\Phi(x^{k+1}_0)\bigr] + \tau\left(\frac{11}{8}\frac{L_f\sigma_g^2}{B} + \frac{3\ell_f\sigma_{g'}^2}{4L_g S}\right)\\
    && - \left(\frac{M}{4} - \frac{2\tau}{b}L_{f\circ g}\right) \sum_{i=0}^{\tau-1} \E\bigl[\bigl\|\xkie-\xki\bigr\|^2\bigr] .
\end{eqnarray*}
The choices of $M\geq 4 L_{f\circ g}$ and $b=2\tau$ ensure $\frac{M}{4}-\frac{2\tau}{b}L_{f\circ g}\geq 0$,
and choices of $B=\frac{11L_f\sigma_g^2}{4\epsilon}$ and $S=\frac{3\ell_f^2\sigma_{g'}^2}{2L_g\epsilon}$ further ensure the constant term to be less than $\tau\epsilon$. Therefore
\begin{equation}
\label{eqn:smooth-expectation-sum-tau}
    \frac{1}{24M}\sum_{i=0}^{\tau-1}\E\bigl[\bigl\|\cGki\bigr\|^2\bigr]
    \leq \E\bigl[\Phi(\xkz)\bigr]-\E\bigl[\Phi(x^{k+1}_0)\bigr] + \tau\epsilon,
\end{equation}
which, upon summing over $k=1,\ldots,K$ and noting the choice of $x^{k^*}_{i^*}$, yields~\eqref{thm:sarah-cvg-infinite-0}.
The sample complexities can be calculated as $KB+K\tau b$.
\end{proof}
 
In Theorem~\ref{theorem:SARAH-cvg-Infinite}, the choices of~$\tau$ and batch sizes all depend on a fixed accuracy~$\epsilon$, which can be hard to determine in advance in many situations, and running more iterations will not improve the solution due to the existence of a $\cO(\epsilon)$ bias term in~\eqref{thm:sarah-cvg-infinite-0}.
Therefore, it would be desirable to develop an algorithm that adaptively chooses the batch sizes to keep improving the accuracy of the solution. Such an adaptive scheme is presented in the following theorem.
 
\begin{theorem}
\label{theorem:SARAH-cvg-Infinite-adp} 
Suppose Assumptions~\ref{assumption:f-cvx}, \ref{assumption:phi-lowerbound}, \ref{assumption:g-lip-infinite} and~\ref{assumption:SVR-PL-f-smooth} hold for problem~\eqref{eqn:composite-expect}. Let the estimates $\tgkz$, $\tJkz$, $\tgki$ and $\tJki$ in Algorithm~\ref{algo:SVR-PL} be given in~\eqref{eqn:g-J-spider-0}-\eqref{eqn:J-spider-infinite}, and we choose $M\geq 4 L_{f\circ g}$.
    Let $\{\epsilon_k\}_{k =1}^{\infty}$ be a sequence of positive real numbers. If we run each epoch of Algorithm~\ref{algo:SVR-PL} for $\tau_k=\epsilon_k^{-1/2}$ iterations, and set the batch sizes to be
\[
    \bigl|\Bkz\bigr|=\left\lceil\frac{11L_f\sigma_g^2}{4\epsilon_k}\right\rceil, \qquad
    \bigl|\Skz\bigr|=\left\lceil\frac{3\ell_f^2\sigma_{g'}^2}{2L_g\epsilon_k}\right\rceil, \qquad
    \bigl|\Bki\bigr|=\bigl|\Ski\bigr|=2 \left\lceil \epsilon_k^{-1/2}\right\rceil, 
    %\qquad i=1,\ldots,\tau-1,
\]
where $i=1,\ldots,\tau-1$,
then we have 
\begin{equation}
\label{thm:sarah-cvg-infinite-adp-1}
 	\frac{1}{20M}\E[\|\cG(x_{i^*}^{k^*})\|^2] 
    ~\leq~ \frac{\Phi(x^1_0) - \Phi(x^*)}{\sum_{k=1}^{K}\tau_k} + \frac{\sum_{k=1}^K\epsilon_k^{1/2}}{\sum_{k=1}^K\tau_k}
\end{equation} 
Specifically, setting $\epsilon_k = k^{-2}$ results in 
\begin{equation}
\label{thm:sarah-cvg-infinite-adp-2}
 	\frac{1}{20M}\E[\|\cG(x_{i^*}^{k^*})\|^2] 
    ~=~ \cO\left(\frac{\ln K}{K^2}\right).
\end{equation}
Consequently, given any $\epsilon>0$, we can set $K=\epsilon^{-1/2}$, which leads to an $\cO\bigl(\epsilon\ln\frac{1}{\epsilon}\bigr)$-stationary solution with total sample complexity of $\cO(\epsilon^{-3/2})$.
\end{theorem}
\begin{proof}
    Note that the inequality~\eqref{eqn:smooth-expectation-sum-tau} still holds but with a specific set of parameters for each~$k$. 
Specifically, we have
\begin{eqnarray*} 
 	\frac{1}{24M}\sum_{i=0}^{\tau_k-1}\E\bigl[\|\cGki\|^2\bigr] 
    ~\leq~ \E\bigl[\Phi(x_0^k)\bigr] - \E\bigl[\Phi(x_0^{k+1})\bigr]+ \tau_k\epsilon_k , \qquad k=1,\ldots,K.
%    = \E[\Phi(x_0^k)] - \E[\Phi(x_{\tau_k}^k)]+ \epsilon_k^{1/2}.
\end{eqnarray*}
Since we choose $\tau_k=\epsilon_k^{-1/2}$, it holds that $\tau_k\epsilon_k=\epsilon_k^{1/2}$. 
Summing this up over~$k$ gives 
\[
 	\frac{1}{24M}\sum_{k=1}^K\sum_{i=0}^{\tau_k-1}\E\bigl[\|\cGki\|^2\bigr] 
    ~\leq~  \Phi(x_0^1) - \Phi_* + \sum_{k=1}^K\epsilon_k^{1/2}.
\] 
Because $x^{k^*}_{i^*}$ is randomly chosen from $\bigl\{\xki\bigr\}_{i=0,\ldots,\tau-1}^{k=1,\ldots,K}$, we conclude~\eqref{thm:sarah-cvg-infinite-adp-1} holds.

If we choose $\epsilon_k = k^{-2}$, then $\tau_k=\epsilon^{-1/2}=k$ and we have 
$$\sum_{k=1}^K\tau_k = \half K(K+1)\quad\mbox{ and }\quad\sum_{k=1}^K\epsilon_k^{1/2} = \sum_{k=1}^k k^{-1} \leq 1 + \int_{1}^{K}z^{-1}dz = 1+\ln K.$$
Substituting the above relationships into~\eqref{thm:sarah-cvg-infinite-adp-1}
yields~\eqref{thm:sarah-cvg-infinite-adp-2}.
The total sample complexity for the $g_\xi$'s for running these $K$ epochs will be 
$$\sum_{k=1}^K\Bigl(|\Bkz| + \tau_k|\cB^k_1|\Bigr) =\cO\left(\sum_{k=1}^{K}\epsilon_k^{-1}\right) = \cO(K^3).$$
Similarly, the sample complexity for the Jacobians is also $\cO(K^3)$. 
Finally by setting $K = \epsilon^{-1/2}$, we will get an $\cO\bigl(\epsilon\ln \frac{1}{\epsilon}\bigr)$-stationary solution with total sample complexity of $\cO(\epsilon^{-3/2})$. 
\end{proof}

\section{Numerical experiments}
\label{sec:experiments}

In this section, we present the numerical experiments of our methods and compare with related methods (following the experiment setup in~\cite{tran2020stochastic}). For the ease of reference, we denote the algorithms in comparison as follows: 
\begin{itemize}
\item PL: the deterministic prox-linear algorithm described by \eqref{eqn:prox-linear}. 
\item S-PL: the mini-batch stochastic prox-linear algorithm in Algorithm \ref{algo:mini-batch-PL}. We note that S-PL coincides with SGN method in the concurrent work \cite{tran2020stochastic}. 
\item SVR-PL: Algorithm~\ref{algo:SVR-PL} where the SVRG estimator is applied and augmented with 1st-order correction technique.
\item Sarah-PL: Algorithm~\ref{algo:SVR-PL} using the SARAH estimator. Specifically, when $f$ is nonsmooth, Sarah-PL overlaps with SGN2 \cite{tran2020stochastic}. 
\item When $f$ is smooth, we also compare with the CIVR method \cite{C-SARAH} and the N-Spider method \cite{NestedSpider} for stochastic composite optimization.
\end{itemize}

\subsection{Nonsmooth nonlinear systems}
\label{subsec:nonsmooth-eq}

In this experiment, we solve the following nonsmooth problem:
$$\minimize_{x\in\R^n}\,\, \Phi(x):= \Big\|\frac{1}{N}\sum_{j=1}^{N}g_j(x)\Big\|_1 + \beta\|x\|_1,$$
where we want to find a sparse $x$ s.t. $\frac{1}{N}\sum_{j=1}^{N}g_j(x)$ is close to 0. Let $a_j\in\R^n$ be the $j$-th data point and $b_j\in\{-1,+1\}$ be the corresponding label. The function $g_j:\R^n\mapsto\R^4$ is defined as 
\begin{eqnarray}
	\label{eqn:exp-g-1}
	\!\!g_j(x) = \begin{bmatrix}
		1-\tanh(z_j)\\ \Big(1-\frac{1}{1+e^{-z_j}}\Big)^2\\ \log\left(1+e^{-z_j}\right) - \log\left(1+e^{-z_j-1}\right) \\ \log\left(1+(z_j-1)^2\right)
	\end{bmatrix}\quad\mbox{with}\quad z_j = b_j\cdot a_j^T x,\quad
\end{eqnarray} 
where each row of $g_j$ corresponds a certain type of binary classification loss, which can be viewed as a mixture of multiple models. We test our methods with the ijcnn1 dataset\footnote{https://www.csie.ntu.edu.tw/~cjlin/libsvmtools/datasets/binary.html} and the MNIST dataset\footnote{http://yann.lecun.com/exdb/mnist/}. For ijcnn1, we randomly extract $N=10000$ data points.
% to construct the objective function. 
For MNIST, we extract $N=10000$ data points of two digits (Figure~\ref{fig:nonsmooth} shows the plots for ``1'' and ``9''). Specifically, each data point $a_j$ in the ijcnn1 dataset is 22 dimensional, we set $\beta = 0$ for ijcnn1 dataset, meaning that we do not require the solution to be sparse. For the MNIST dataset, each $a_j$ are 784 dimensional where most entries are 0. In this case, we set $\beta = N^{-1}$ as the sparsity penalty parameter. 

In the experiment, we test PL, S-PL/SGN, SVR-PL and Sarah-PL/SGN2 algorithms. 
For SVR-PL and Sarah-PL, we estimate $g$ and $g'$ with the mini-batch sizes suggested in Remark \ref{remark:dependent-batches}. Specifically, for SVR-PL, we choose $|\cS_i^k| = |\cB_i^k| = \lceil cN^{4/5}\rceil$. For Sarah-PL, we choose $\epsilon = 10^{-2}$, therefore we choose the large batches to be $|\cS_0^k| = |\cB_0^k|= \epsilon^{-2} = N$. For this finite sum problem we slightly revise the Sarah-PL such that $\tilde g_0^k = g(x_0^k)$ and $\tilde J_0^k = g'(x_0^k)$ and we  set $|\cS_i^k| = |\cB_i^k| = \lceil c\epsilon^{-3/2}\rceil$ for $i>0$. For both SVR-PL and Sarah-PL, $c$ is set to be $c=0.1$ after tuning from the set
$\{0.01,0.05,0.1,0.5,1,2\}$. 
For S-PL, the batch size is set to be $500$. 
For all methods, we select the best performing $M$ from the discrete range $\{1,5,10,20,40,60,80,100\}$. 
For ijcnn1 dataset, $M = 1$ works best for all methods;  For MNIST dataset, $M = 40$ works best for all methods. 
All methods start from the initial solution $x=0$.  

The results are shown in Figure \ref{fig:nonsmooth}, where each curve is plotted by averaging 5 rounds of running an algorithm. 
We can see that in terms of sample complexity, all stochastic methods significantly outperforms the deterministic PL algorithm.
Among the stochastic methods, SVR-PL and Sarah-PL perform better than S-PL (mini-batch only). 
Sarah-PL performs the best for this particular experiment, benefiting from 
using $|\cS_0^k| = |\cB_0^k|= N$ in the finite-sum setting, even though we do not have theory to support its advantage.

\begin{figure}[t]
	\centering 
	\includegraphics[width=0.45\linewidth]{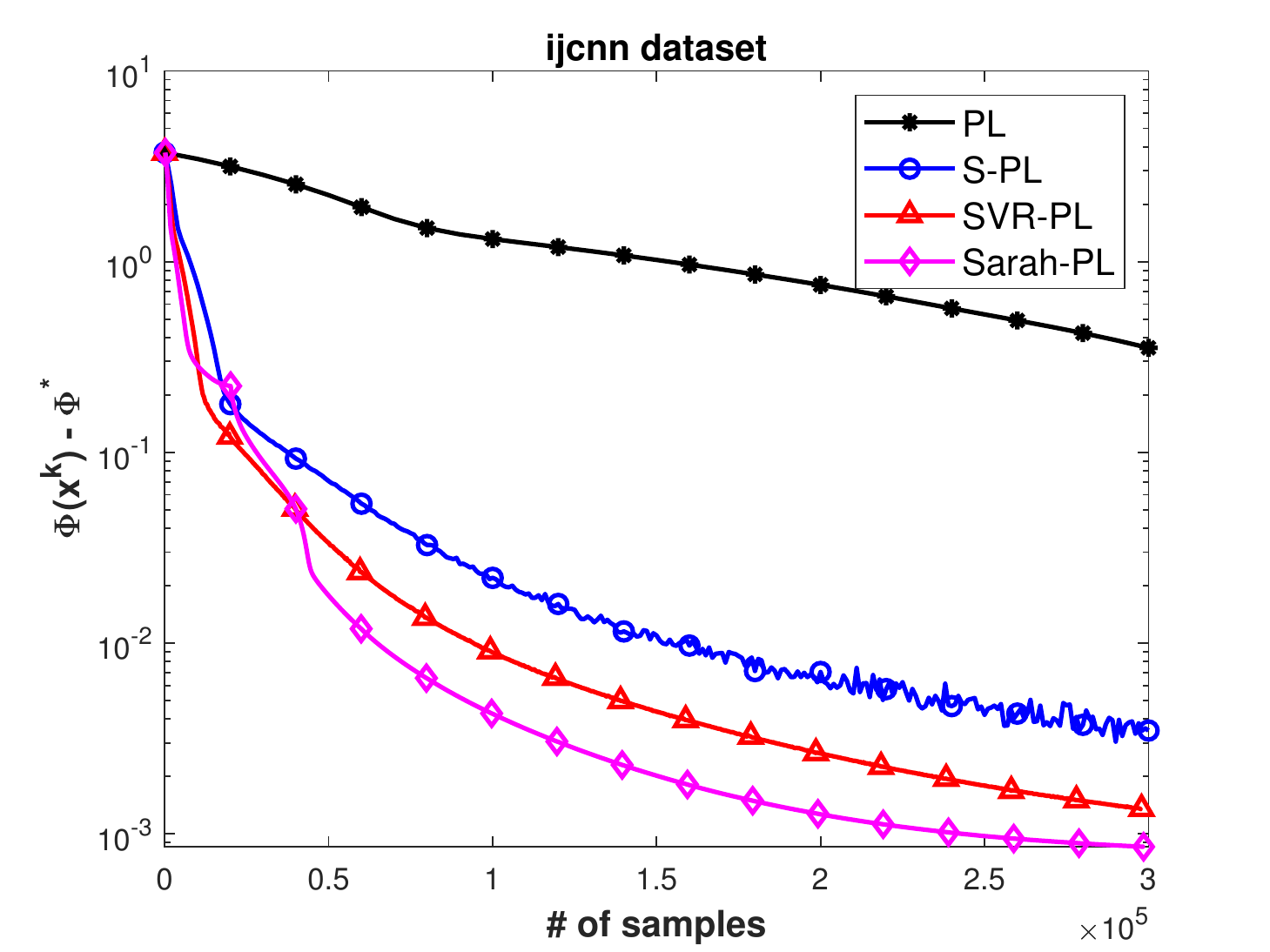} \quad
	\includegraphics[width=0.45\linewidth]{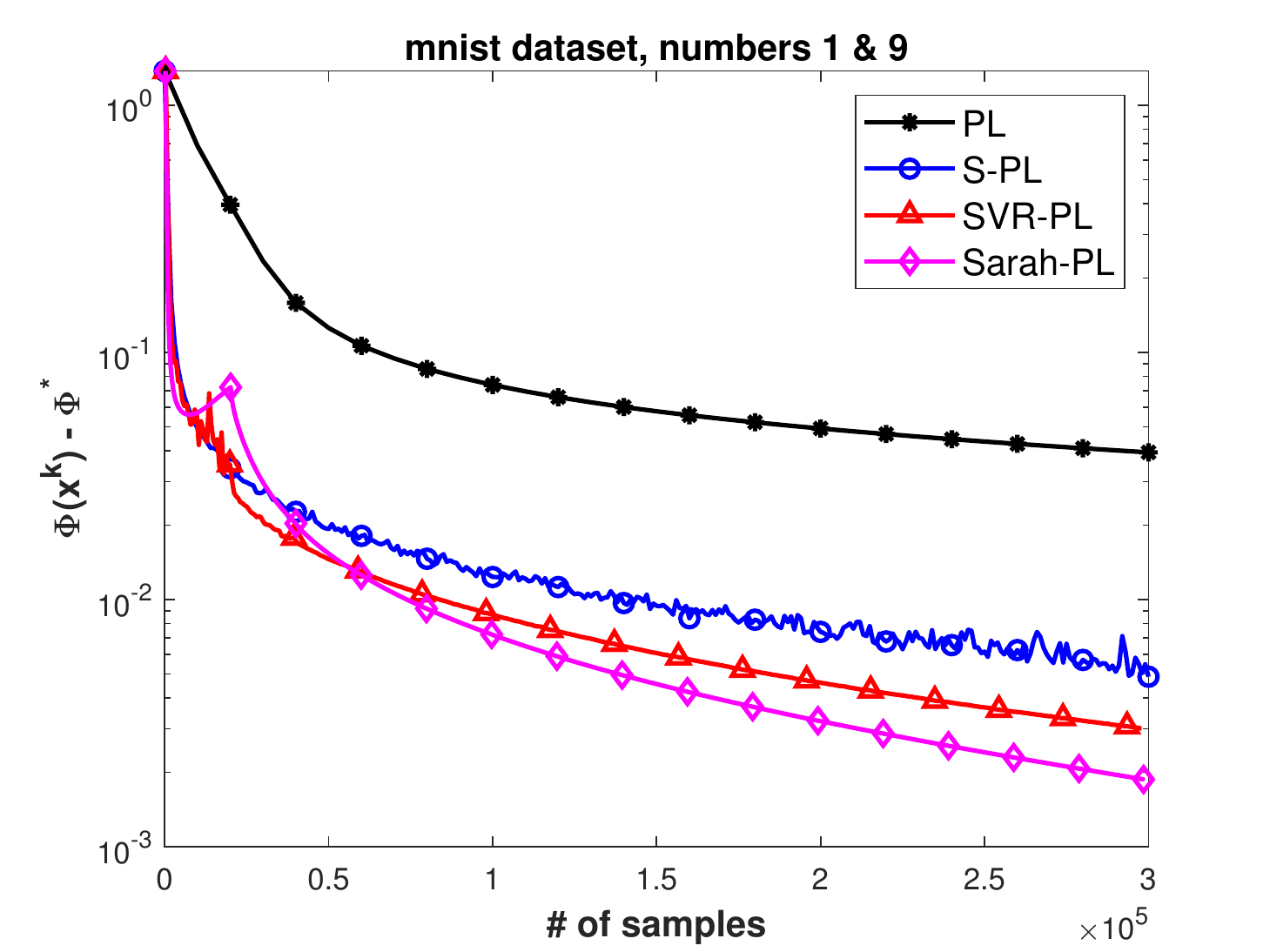} \\
	\includegraphics[width=0.45\linewidth]{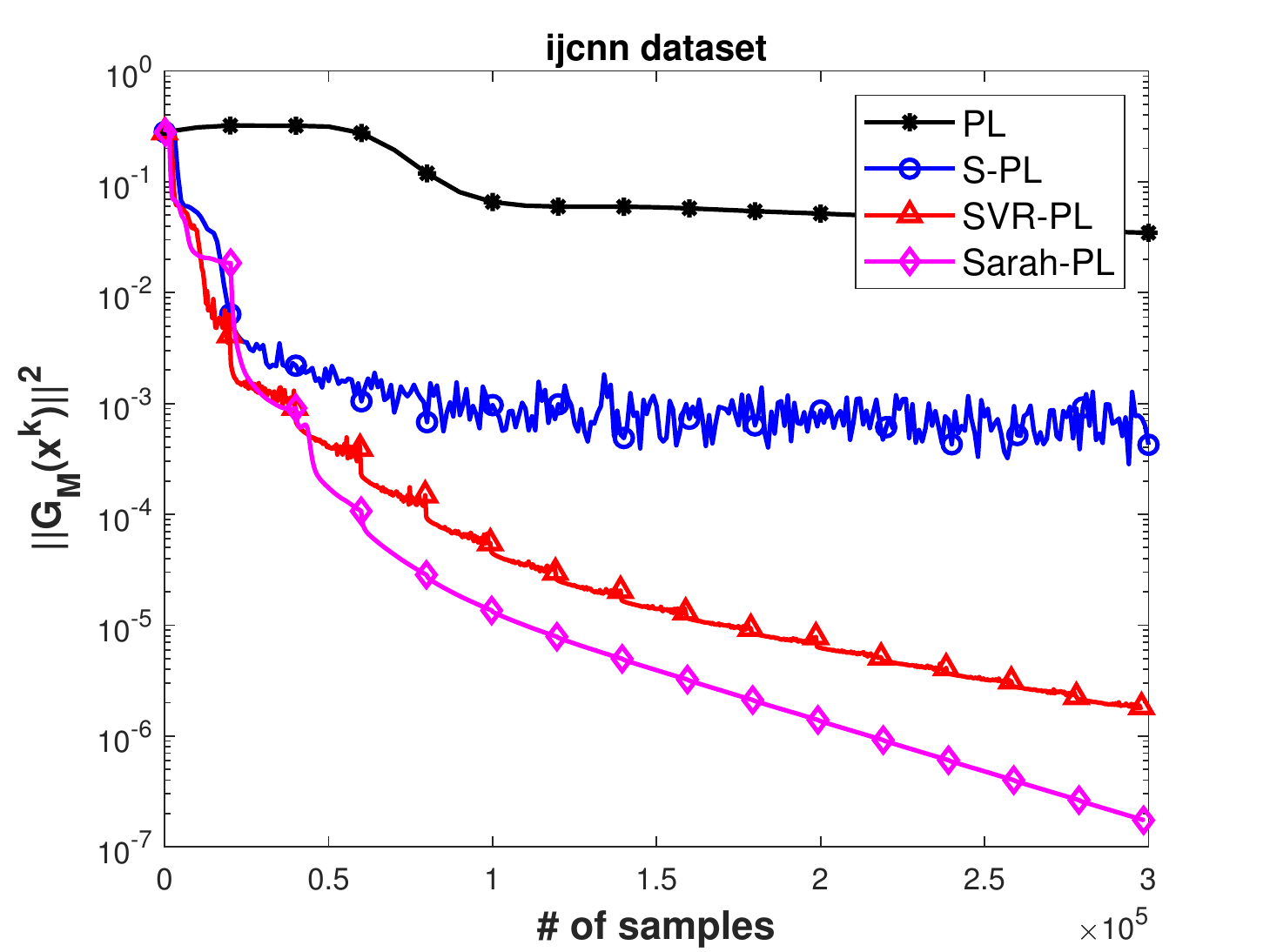} \quad 
	\includegraphics[width=0.45\linewidth]{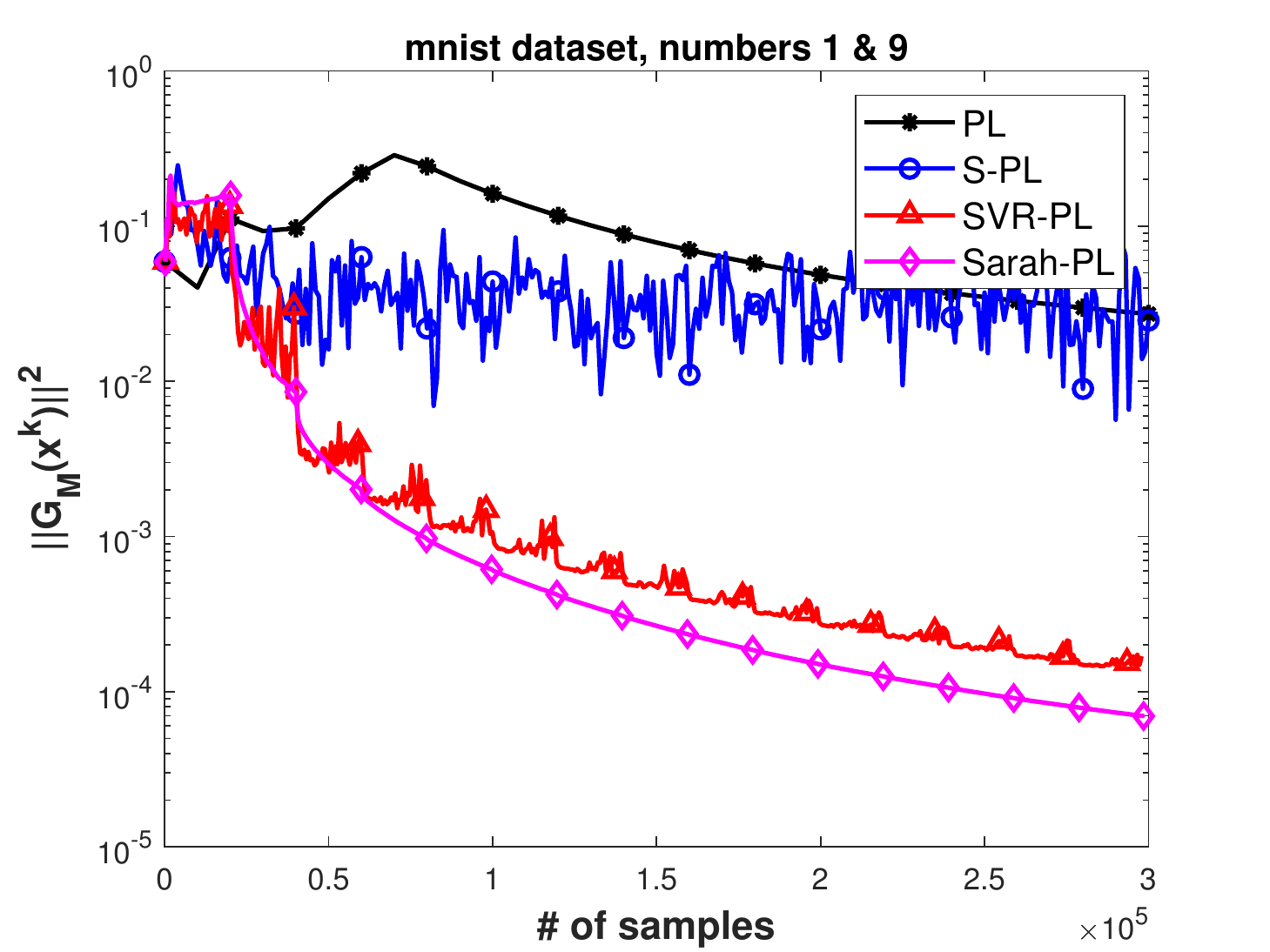}
	\caption{Comparison of the deterministic and stochastic prox-linear methods for nonsmooth composite optimization: the left column is for the ijcnn1 dataset and the right column is for the MNIST dataset (digits ``1'' and ``9'').
The first row shows the decrease of objective gap versus number of samples
(where $\Phi^*$ is approximated by collecting the lowest value after running all algorithms for a much longer time). The second row shows squared norm of the (exact) gradient mapping (computed off-line using the full dataset).}
	\label{fig:nonsmooth}
\end{figure}

\begin{figure}[t]
	\centering 
	\includegraphics[width=0.45\linewidth]{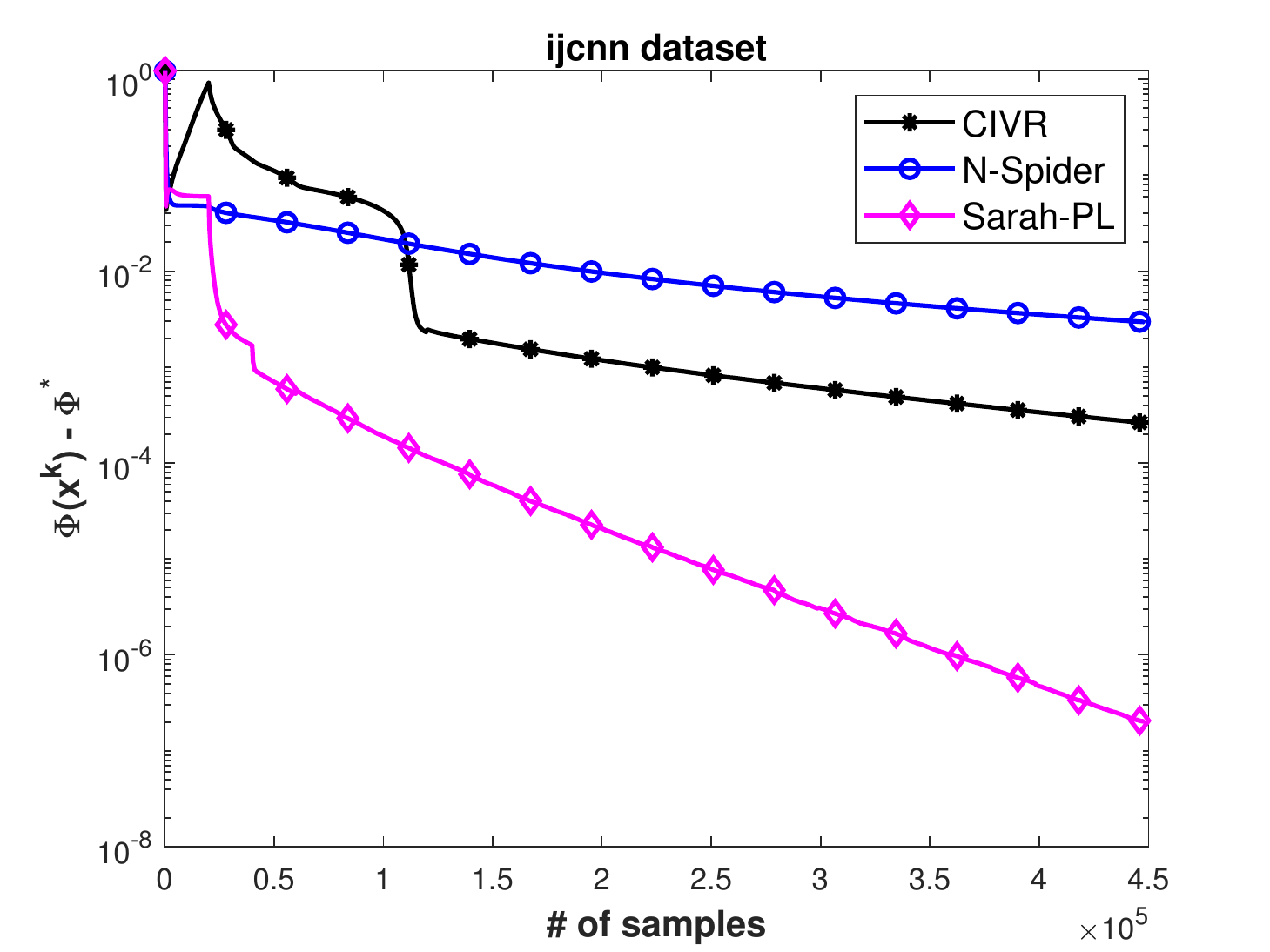} \quad 
	\includegraphics[width=0.45\linewidth]{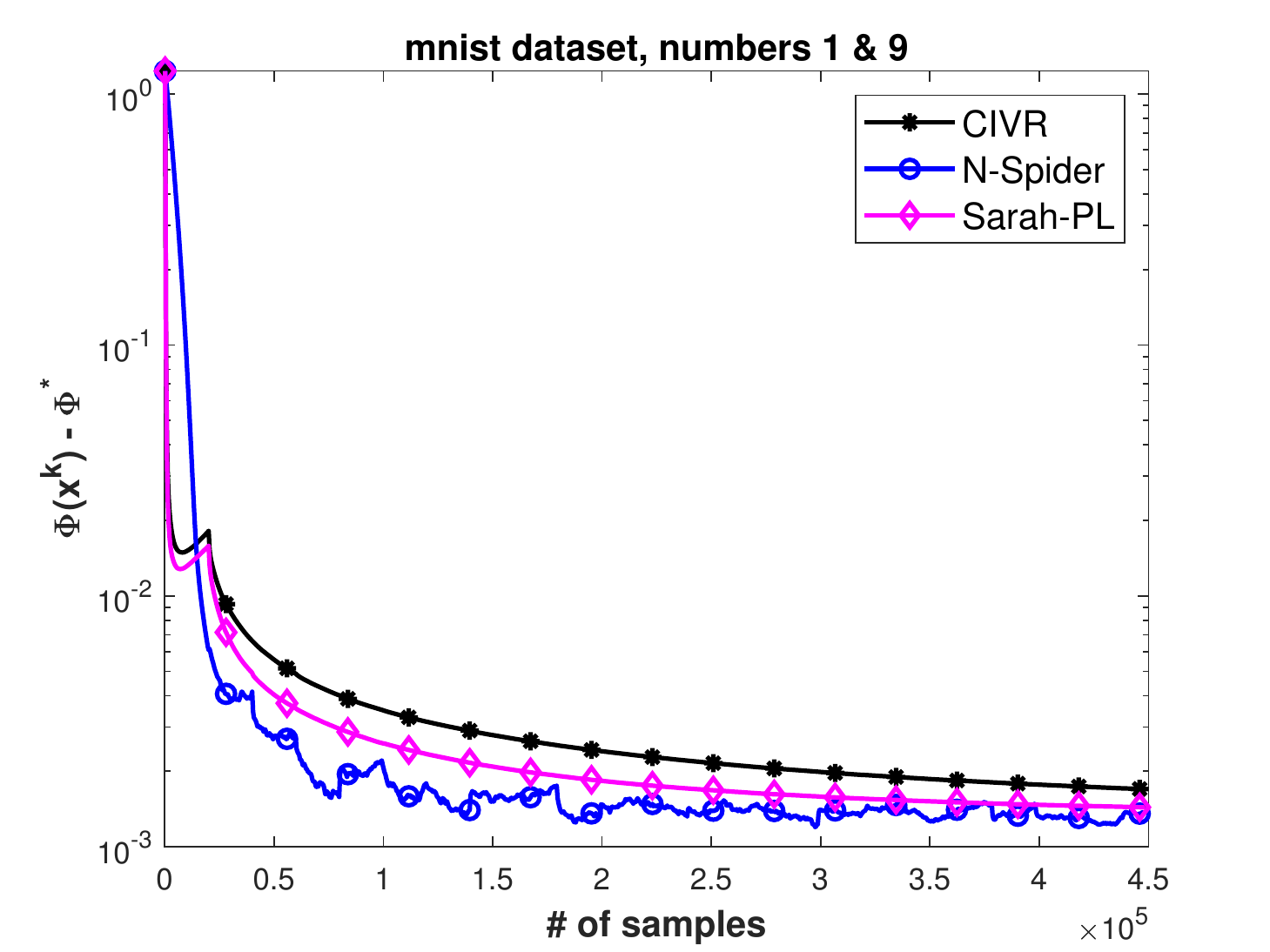} \\
	\includegraphics[width=0.45\linewidth]{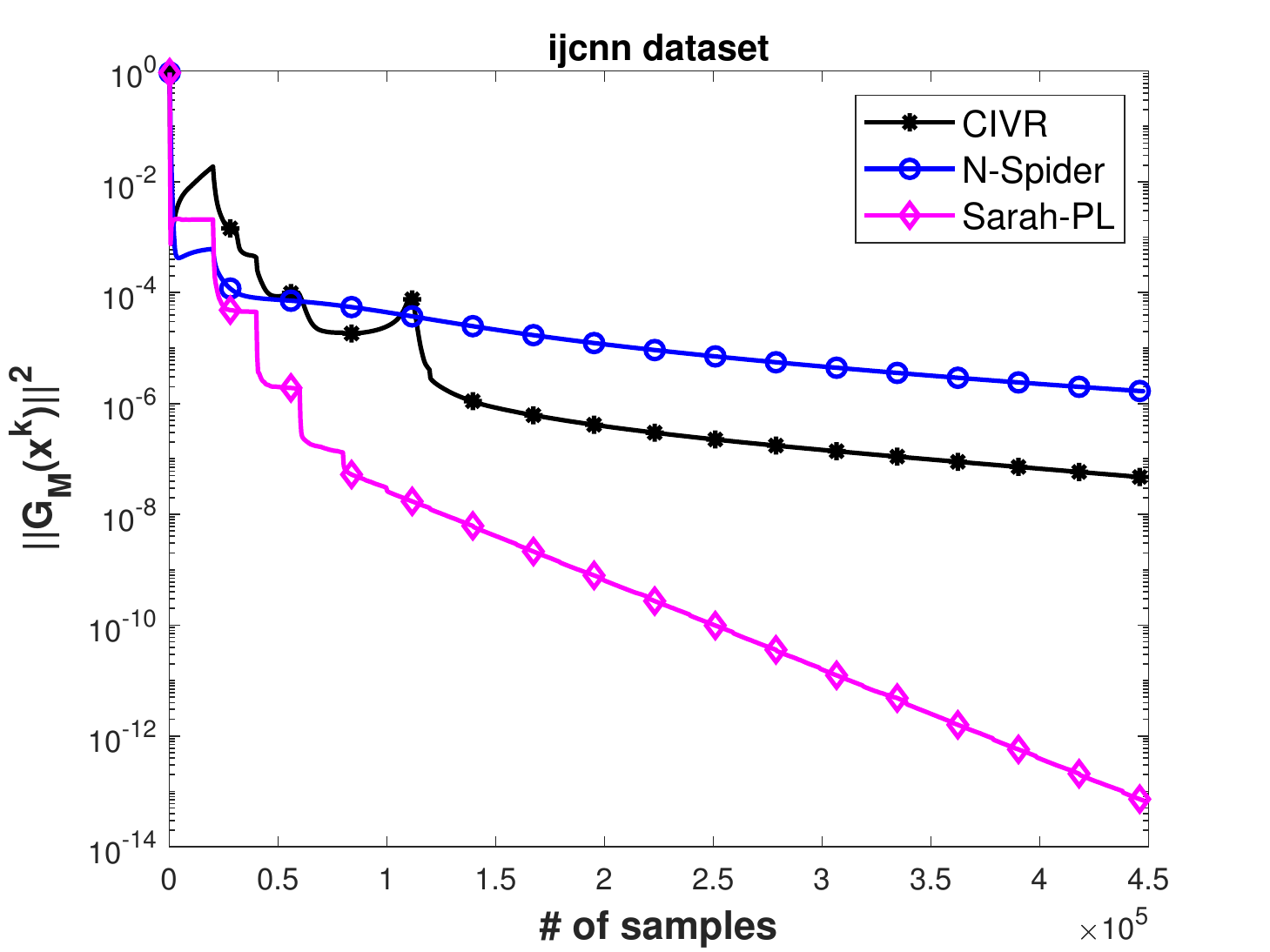} \quad 
	\includegraphics[width=0.45\linewidth]{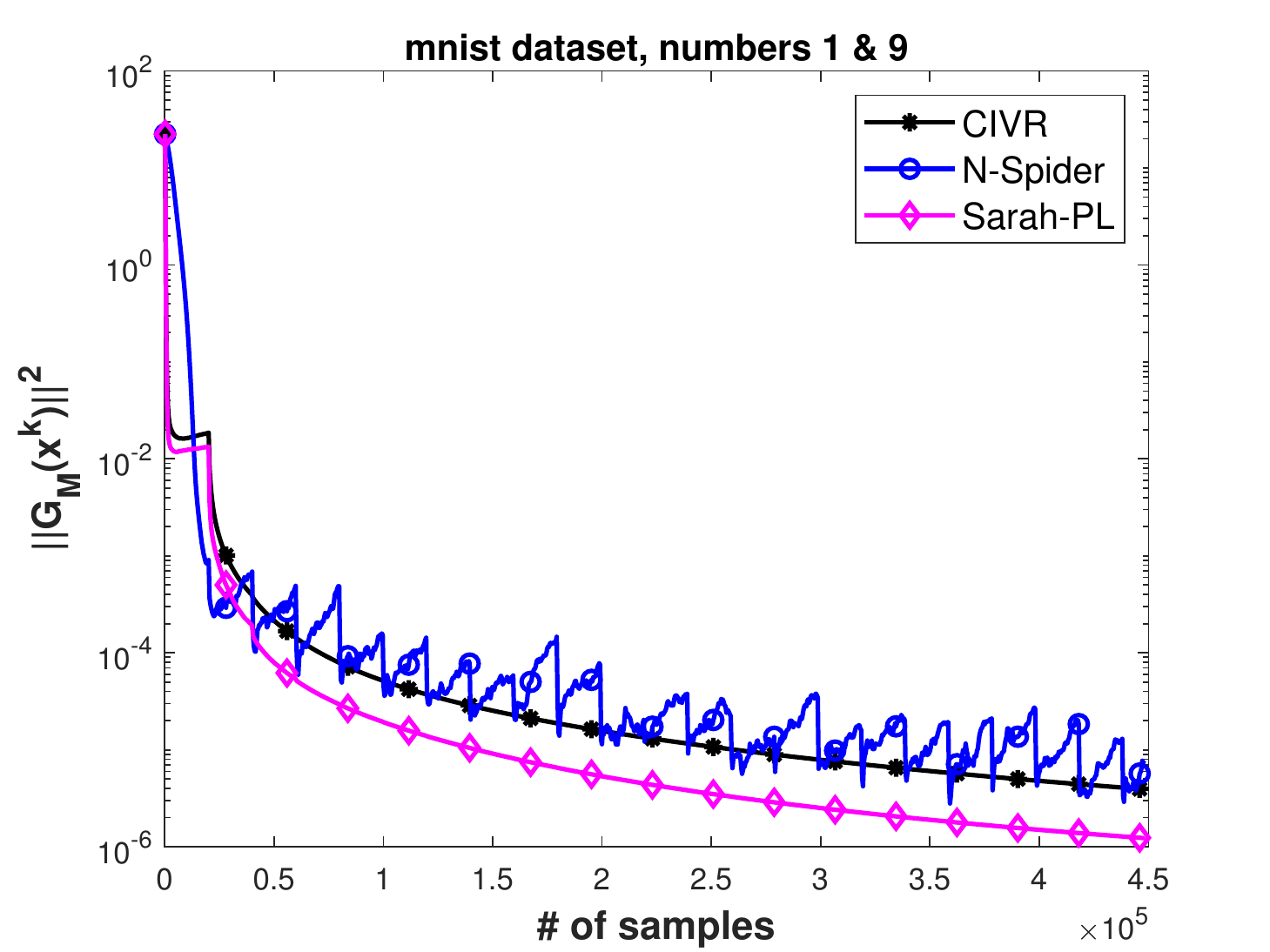}
	\caption{Comparison of stochastic variance-reduction methods for smooth composite optimization: the left column is for the ijcnn1 dataset and the right column is for the MNIST dataset (digits ``1'' and ``9'').
The first row shows objective gap versus number of samples
(where $\Phi^*$ is approximated by collecting the lowest value after running all algorithms for a much longer time). The second row shows squared norm of the (exact) gradient mapping (computed off-line using the full dataset).}
	\label{fig:smooth}
\end{figure}

\subsection{Smooth nonlinear systems}
\label{subsec:Exp-smooth}
In this experiment, we solve the following smooth problem:
$$\minimize_{x\in\R^n}\,\, \Phi(x):= \Big\|\frac{1}{N}\sum_{j=1}^{N}g_j(x)\Big\|^2,$$
where $g_j:\R^n\mapsto\R^4$ is defined by \eqref{eqn:exp-g-1}. 
We compare Sarah-PL, CIVR \cite{C-SARAH}, and the N-Spider algorithm \cite{NestedSpider}. For all three methods, the batch sizes are set to be $\bigl\lceil\sqrt{N}\bigr\rceil$. For N-Spider, we set $\epsilon^k_i = \frac{10}{1+k}$ for ijcnn1 and $\epsilon^k_i = \frac{10^{-2}}{1+k}$ for MNIST after some tuning. For both Sarah-PL and CIVR, their parameter $M$ or step size $\eta = M^{-1}$ are chosen from the set $\{0.1,0.5,1,5,10,20,40,60\}$. For ijcnn1, Sarah-PL works best with $M = 0.1$ and CIVR works best with $\eta = 0.5^{-1}$. For MNIST, Sarah-PL works best with $M = 10$ and CIVR works best with $\eta = 20^{-1}$. 

The results are shown in Figure \ref{fig:smooth}, where each curve is plotted by averaging 5 rounds of running an algorithm. 
For ijcnn1, Sarah-PL significantly outperforms CIVR and N-Spider, demonstrating the potential advantage of prox-linear algorithms over chain-rule based methods %(see discussions in Section~\ref{sec:discussions}).
(both with variance reduction).
For MNIST, all three methods performs similarly.

\begin{figure}[t]
	\centering 
	\includegraphics[width=0.45\linewidth]{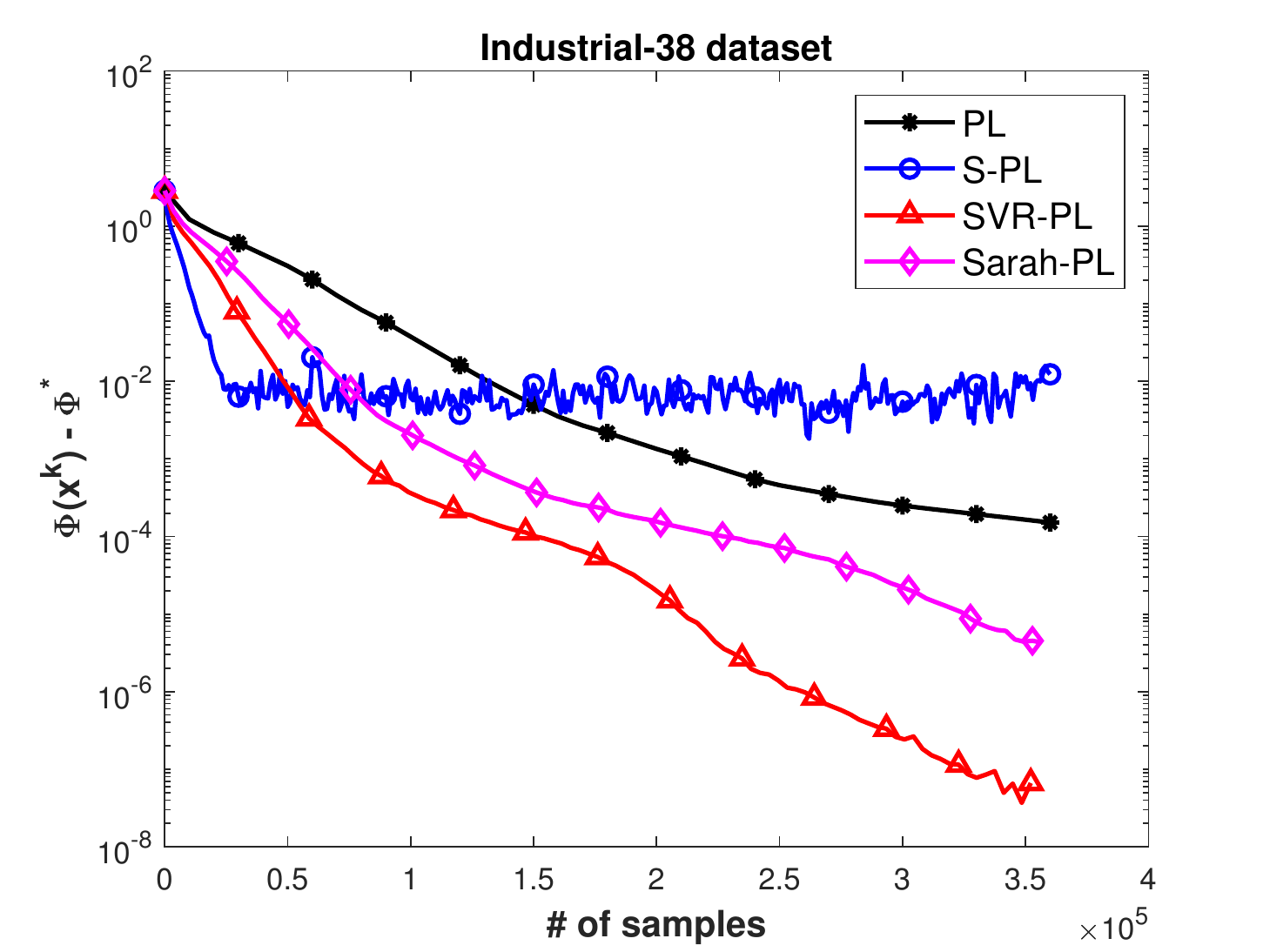} \quad
	\includegraphics[width=0.45\linewidth]{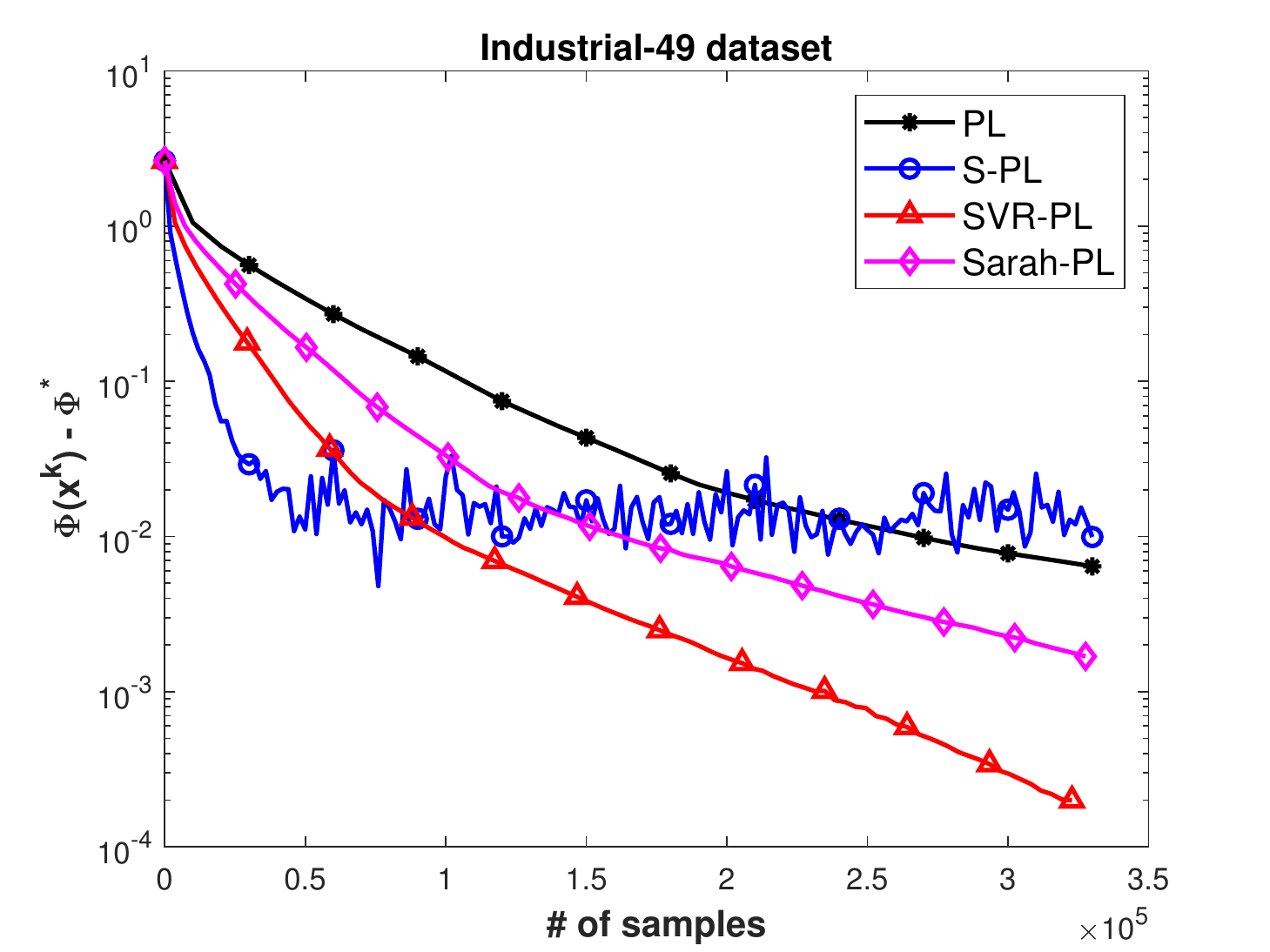} \\
	\includegraphics[width=0.45\linewidth]{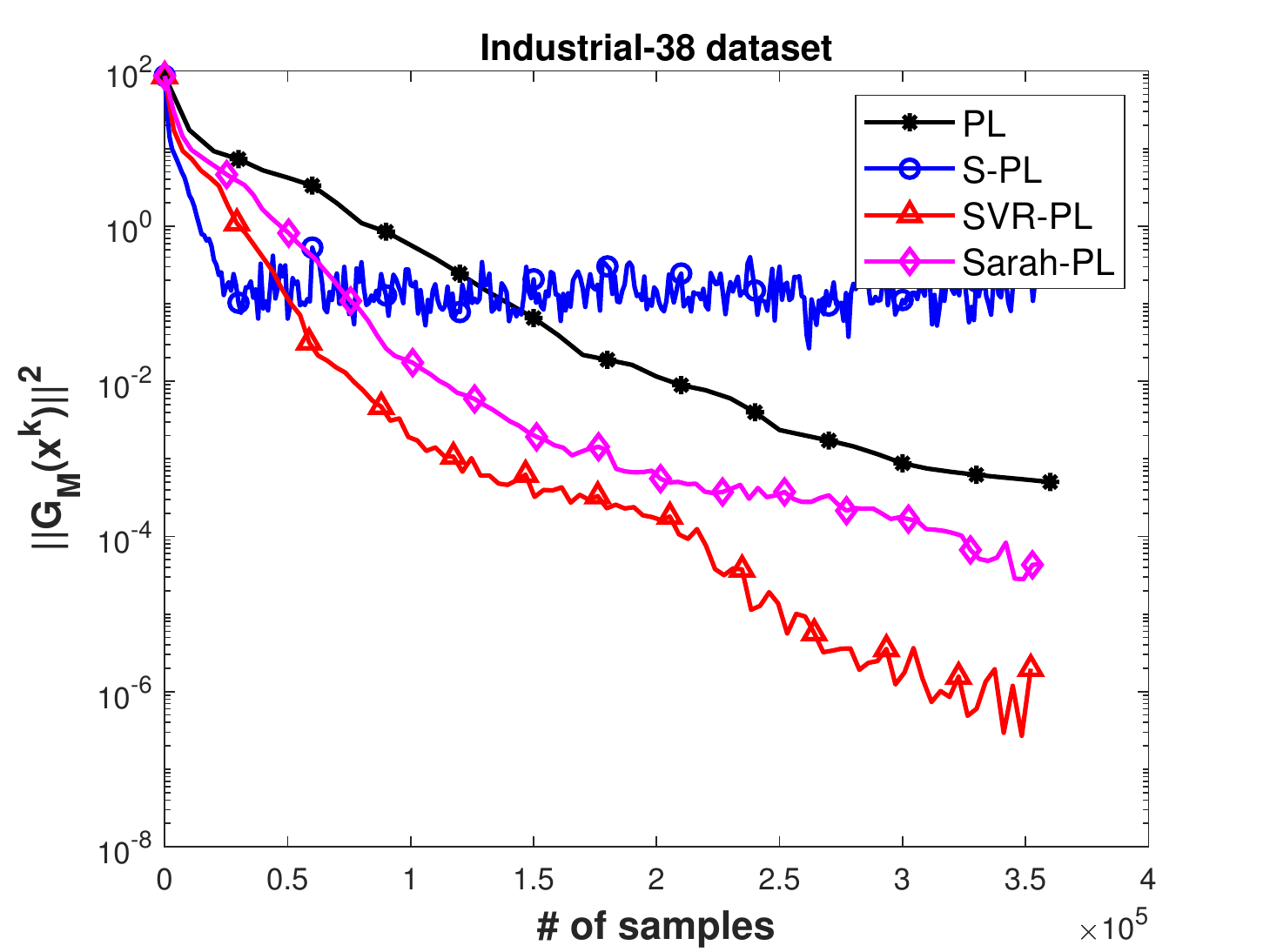} \quad 
	\includegraphics[width=0.45\linewidth]{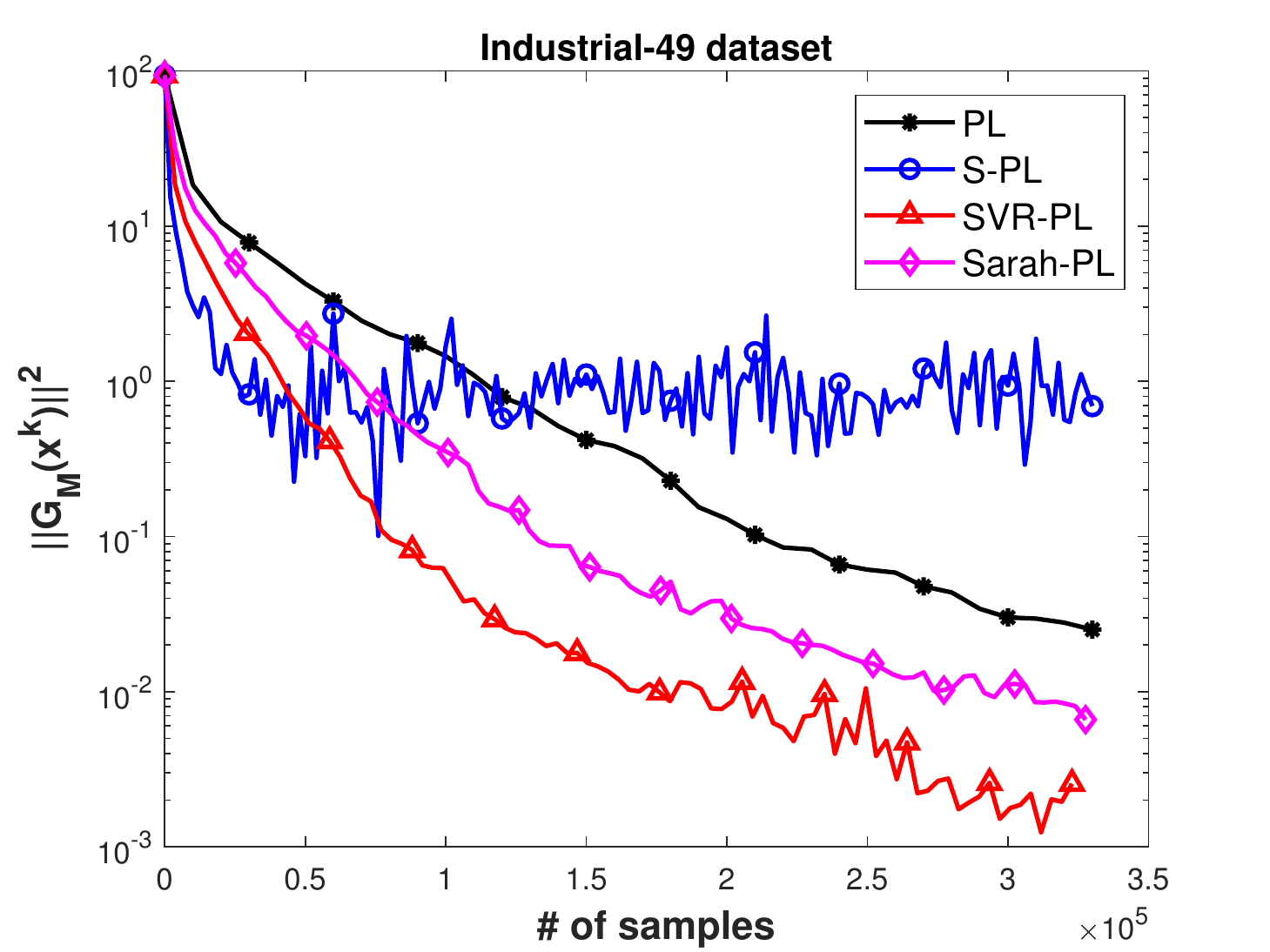} 
	\caption{Comparison of different prox-linear algorithms for constrained stochastic optimization through penalty formulation: the left column is for the Industrial-38 dataset and the right column is for the Industrial-49 dataset.
The first row shows objective gap versus number of samples
%(where $\Phi^*$ is approximated by collecting the lowest value after running all algorithms for a much longer time). 
and the second row shows squared norm of the (exact) gradient mapping.}
	\label{fig:penalty}
\end{figure}

\subsection{Constrained stochastic optimization through penalty method}
%\subsection{Risk-sensitive portfolio optimization}
\label{subsec:portfolio}
We consider a risk-sensitive portfolio optimization problem.  
Let $r_i\in\R^d$ be the vector of expected reward of~$d$ stocks at time period $i$, for $i = 1,2,...,N$. 
The problem of maximizing the expected total reward across~$N$ periods,
with a constraint on the conditional value at risk (CVaR) is formulated as
\cite{rockafellar2000optimization,lan2020algorithms}
\begin{eqnarray*}
	\label{prob:portfolio}
	\maximize_{x\in\Delta_d,\tau\in\R}~\Big(\frac{1}{N}\sum_{i=1}^{N}r_i\Big)^Tx\qquad\mathrm{s.t.} \qquad\tau + \frac{1}{\beta N}\sum_{i=1}^{N}\max\bigl\{-r_i^Tx-\tau,0\bigr\}\leq0, 
\end{eqnarray*}
where $\Delta_d:=\bigl\{x:x\geq0, \,\, \sum_{j=1}^{d}x_j = 1\bigr\}$ is the probability simplex. 
Using the exact penalty method (Section~\ref{sec:examples}), this problem can be reformulated as 
%\label{prob:portfolio'}
\[
\minimize_{x\in\Delta_d,\tau\in\R}~ -\Big(\frac{1}{N}\sum_{i=1}^{N}r_i\Big)^T x + \rho\cdot\max\Big\{0,\tau + \frac{1}{\beta N}\sum_{i=1}^{N}\max\{-r_i^Tx-\tau,0\}\Big\} .
\]
Following the suggestion of \cite{tran2020stochastic}, the nonsmooth term $\beta^{-1}\cdot\max\{-r_i^Tx-\tau,0\}$ is smoothed as $\frac{1}{2\beta}\left(\sqrt{(r_i^Tx+\tau)^2+\gamma^2} - r_i^Tx - \tau - \gamma\right)$.

In this experiment, we test different methods on the Industrial-38 and the Industrial-49 dataset\footnote{http://mba.tuck.dartmouth.edu/pages/faculty/ken.french/data\_library.html}. From each dataset, $N=10000$ data points are extracted for the experiment. The parameters in the problem formulation are set to be $\beta = 10^{-1}$, $\rho = 5$, and $\gamma = 10^{-3}$. For algorithmic parameters, their tuning process is the same as that described in Section \ref{subsec:nonsmooth-eq}. The following results are obtained. 
In Industrial-38 dataset, $M = 40, 60, 40, 60$ works best for PL, S-PL,SVR-PL and Sarah-PL respectively; For S-PL, the batch size is chosen to be 1000; For both SVR-PL and Sarah-PL, the batch sizes are the same as those in Section \ref{subsec:nonsmooth-eq} with $c=2$. 

Figure \ref{fig:penalty} shows the results, again averaged over 5 runs of each algorithm. 
In this experiment, SVR-PL performs the best. It is worth noting that 
although S-PL has fast convergence in the initial stage, it stagnates at a 
relatively high error floor. 
SVR-PL and Sarah-PL reach higher accuracy due to their advanced variance-reduction schemes.

\section{Discussions}
\label{sec:discussions}

In this paper, we have mostly relied on the SARAH/\textsc{Spider} estimators for variance reduction, except that for the nonsmooth and finite-average case (Section~\ref{sec:nonsmooth-finite}) we used a modified SVRG estimator with first-order correction. 
If we use the SVRG type of estimators for other cases, then the resulting sample complexities are suboptimal. More specifically, when~$f$ is smooth, we have derived sample complexity of $\cO(N+N^{2/3}\epsilon^{-1})$ and $\cO(\epsilon^{-5/3})$ for the cases of~$g$ being a finite average and a general expectation respectively. They are inferior compared to the $\cO(N+\sqrt{N}\epsilon^{-1})$ and $\cO(\epsilon^{-3/2})$ bounds using the SARAH/\textsc{Spider} estimators obtained in Sections~\ref{sec:smooth-finite} and~\ref{sec:smooth-expect}.

The sample complexities of our methods for smooth~$f$ are the same as the stochastic gradient descent type of methods that use the chain-rule to construct gradient estimators \cite{C-SARAH,NestedSpider}. However, it is often observed that algorithms based on proximal mappings can be more efficient than those based on gradients in practice (see e.g., \cite{AsiDuchi2019Truncated,AsiDuchi2019StoProxPoint,DamekDima2019ModelBased}). 
Here we shed more light from a theoretical perspective. 
%To see the intuition behind, 
Consider the least squares problem of minimizing $F(x):=\frac{1}{2}\|g(x)\|^2$, where $g(x) = \frac{1}{N}\sum_{i=1}^Ng_i(x)$. Given any SARAH/\textsc{Spider} variance reduced estimator $\tilde{g}^k_i$ and $\tilde{J}^k_i$, our proxi-linear scheme construct the update as 
$$x^k_{i+1} = x^k_i - \big(M\cdot I +  [\tilde{J}^k_i]^T\tilde{J}^k_i\big)^{-1}[\tilde{J}^k_i]^T\tilde{g}^k_i,$$
which is a \emph{damped Gauss-Newton} iteration.
% (or Levenberg-Marquardt method). 
Note that if $g(x^*) \approx 0$, then $\nabla^2F(x^*) \approx [g'(x^*)]^Tg'(x^*)$ (see \cite{NocedalWright2006book}). This indicates that the Gauss-Newton matrix $[\tilde{J}^k_i]^T\tilde{J}^k_i$ becomes a better approximation of the Hessian $\nabla^2F(x^k_i)$ as $x^k_i$ moves closer to $x^*$.
Therefore, prox-linear based methods (Gauss-Newton especially) can take advantage of the second-order information whenever possible, while chain-rule based gradient methods cannot.

It is worth noting that both SVRG and SARAH/\textsc{Spider} schemes need a large sample batch at the beginning of each epoch, and slightly smaller sample batches in later iterations. However, under many circumstances it is more preferable if constant small batches are taken in each iteration. Recently, a STOchastic Recursive Momentum (STORM) variance reduction scheme that takes one sample per iteration has been proposed to solve smooth stochastic programming problem \cite{cutkosky2019momentum}, and has been extended to a distributionally robust optimization (DRO) problem of form \eqref{eqn:composite-expect} with $f$ being smooth \cite{qi2020practical}. An optimal $\cO(\epsilon^{-3/2})$ sample complexity is achieved in these works. However, we were not able to extend the STORM technique to problems with nonsmooth $f$. Deriving an algorithm with (constant) small mini-batch sizes for problems \eqref{eqn:composite-finite} and \eqref{eqn:composite-expect} with nonsmooth~$f$ remains open.

\section*{Acknowledgments}
The authors thank Dmitriy Drusvyatskiy for contributing the example of truncated stochastic gradient method in Section~\ref{sec:examples}.
We are also grateful to the two anonymous referees for their helpful comments and suggestions.

\bibliographystyle{plain}
\bibliography{SVRPL}

\end{document}